\documentclass[amscd,amssymb,11pt]{amsart}

\usepackage{amssymb}
\usepackage[all]{xy}

\newtheorem{thm}{Theorem}[section]
\newtheorem{lem}[thm]{Lemma}
\newtheorem{cor}[thm]{Corollary}
\newtheorem{prop}[thm]{Proposition}

\theoremstyle{definition}

\newtheorem{defn}[thm]{Definition}
\newtheorem{defns}[thm]{Definitions}

\newtheorem{prob}[thm]{Problem}
\newtheorem{notation}[thm]{Notation}
\newtheorem{ex}[thm]{Example}
\newtheorem{exs}[thm]{Examples}

\theoremstyle{remark}

\newtheorem{rem}[thm]{Remark}

\numberwithin{equation}{section}

\newcommand{\thmref}[1]{Theorem~\ref{#1}}
\newcommand{\corref}[1]{Corollary~\ref{#1}}
\newcommand{\secref}[1]{\S\ref{#1}}

\newcommand{\propref}[1]{Proposition~\ref{#1}}
\newcommand{\lemref}[1]{Lemma~\ref{#1}}

\newcommand{\remref}[1]{Remark~\ref{#1}}

\newcommand{\hocolim}{\operatorname*{hocolim}}
\newcommand{\holim}{\operatorname*{holim}}
\newcommand{\colim}{\operatorname*{colim}}

\newcommand{\hofib}{\operatorname*{hofib}}
\newcommand{\hocofib}{\operatorname*{hocofib}}
\newcommand{\Hom}{\operatorname{Hom}}

\newcommand{\Mor}{\operatorname{Mor}}
\newcommand{\Map}{\operatorname{Map}}

\newcommand{\A}{{\mathcal  A}}

\newcommand{\C}{{\mathcal  C}}
\newcommand{\Sp}{{\mathcal  S}}

\newcommand{\T}{{\mathcal  T}}

\newcommand{\U}{{\mathcal  U}}
\newcommand{\V}{{\mathcal  V}}
\newcommand{\D}{{\mathcal  D}}

\newcommand{\sL}{{\mathcal  L}}

\newcommand{\M}{{\mathcal  M}}
\newcommand{\HM}{{\mathcal  H \mathcal M}}

\newcommand{\Qq}{{\mathcal Q}}
\newcommand{\QM}{{\mathcal Q \mathcal M}}
\newcommand{\QU}{{\mathcal Q \mathcal U}}
\newcommand{\HQM}{{\mathcal  H \mathcal Q \mathcal M}}
\newcommand{\HQU}{{\mathcal  H \mathcal Q \mathcal U}}

\newcommand{\Ab}{{\mathcal Ab}}

\newcommand{\Z}{{\mathbb  Z}}

\newcommand{\R}{{\mathbb  R}}

\newcommand{\bL}{{\mathbb L}}

\newcommand{\Sinfty}{\Sigma^{\infty}}
\newcommand{\Oinfty}{\Omega^{\infty}}

\newcommand{\sm}{\wedge}

\newcommand{\ra}{\rightarrow}
\newcommand{\xra}{\xrightarrow}
\newcommand{\la}{\leftarrow}
\newcommand{\xla}{\xleftarrow}

\newcommand{\hra}{\hookrightarrow}

\newcommand{\coker}{\operatorname{coker}}
\newcommand{\im}{\operatorname{im}}

\newdir{ (}{{}*!/-5pt/@^{(}}

\begin{document}

\title[mod 2 homology of infinite loopspaces]{The mod 2 homology of infinite loopspaces}

\author[Kuhn]{Nicholas J.~Kuhn}
\address{Department of Mathematics \\ University of Virginia \\ Charlottesville, VA 22904}
\email{njk4x@virginia.edu}
\thanks{This research was partially supported by National Science Foundation grant 0967649.}

\author[McCarty]{Jason B.~McCarty}
\address{Department of Mathematics \\ Indiana University of Pennsylvania \\ Indiana, PA 15705}
\email{J.B.McCarty@iup.edu}
\thanks{}
\date{September 30, 2012.}


\begin{abstract}    
Applying mod 2 homology to the Goodwillie tower of the functor sending a spectrum $X$ to the suspension spectrum of its 0th space, leads to a spectral sequence for computing $H_*(\Oinfty X;\Z/2)$, which converges strongly when $X$ is 0--connected.   The $E^1$ term is the homology of the extended powers of $X$, and thus is a well known functor of $H_*(X;\Z/2)$, including structure as a bigraded Hopf algebra, a right module over the mod 2 Steenrod algebra $\A$, and a left module over the Dyer--Lashof operations.  This paper is an investigation of how this structure is transformed through the spectral sequence.

Hopf algebra considerations show that all pages of the spectral sequence are primitively generated, with primitives equal to a subquotient of the primitives in $E^1$.

We use an operad action on the tower, and the Tate construction, to determine how Dyer--Lashof operations act on the spectral sequence.  In particular, $E^{\infty}$ has Dyer--Lashof operations induced from those on $E^1$.

We use our spectral sequence Dyer--Lashof operations to determine differentials that hold for any spectrum $X$. The formulae for these universal differentials then lead us to construct an algebraic spectral sequence depending functorially on an $\A$--module $M$.  The topological spectral sequence for $X$  agrees with the algebraic spectral sequence for $H_*(X;\Z/2)$ for many spectra $X$, including suspension spectra and almost all Eilenberg--MacLane spectra.  The $E^{\infty}$ term of the algebraic spectral sequence has form and structure similar to $E^1$, but now the right $\A$--module structure is unstable. Our explicit formula involves the derived functors of destabilization as studied in the 1980's by W.~Singer, J.~Lannes and S.~Zarati, and P.~Goerss.
\end{abstract}

\maketitle


\section{Introduction and main results} \label{introduction}

An infinite loopspace is a space of the form $\Oinfty X$, the 0th space of a fibrant spectrum $X$.  Thus $X$ consists of a sequence of spaces $X_0, X_1, X_2, \dots$, together with homotopy equivalences $X_n \xra{\sim} \Omega X_{n+1}$, and $\Oinfty X = X_0$.

The homology of $X$ is defined by letting $\displaystyle H_*(X) = \colim_n H_{*+n}(X_n)$.
We let all homology be with mod 2 coefficients,  and consider the following basic problem.

\begin{prob} How can one compute $H_*(\Oinfty X)$ from knowledge of $H_*(X)$?
\end{prob}

The graded vector space $H_*(X)$ has a minimum of extra structure: it is an object in $\M$, the category of locally finite right modules over the mod 2 Steenrod algebra $\A$. By contrast, the structure of $H_*(\Oinfty X)$ is much, much richer: it is an object in the category $\HQU$ of restricted Hopf algebras in the abelian category of left modules over the Dyer--Lashof algebra with compatible unstable right $\A$--module structure.

An ideal solution to our problem would be to describe a functor from $\M$ to $\HQU$ whose value on $H_*(X)$ would be $H_*(\Oinfty X)$.  It is not a surprise that such a functor doesn't exist, and we will see examples illustrating this.  However, one punchline of this paper is that one can come surprisingly close.

Our method is to carefully study the left half plane spectral sequence $\{E_{*,*}^r(X)\}$ associated to Goodwillie tower of the functor $X \rightsquigarrow \Sinfty \Oinfty X$.  This converges strongly to $H_*(\Oinfty X)$ when $X$ is $0$--connected, and has an $E^1$ term that is a well known functor of $H_*(X)$, including structure as a primitively generated bigraded Hopf algebra, with Steenrod operations acting vertically on the right, and Dyer--Lashof operations acting on the left and doubling horizontal grading.

It is formal that pages of the spectral sequence will again be primitively generated bigraded Hopf algebras equipped with Steenrod operations, but we prove a more subtle phenomenon: for all $r$, the bigraded module of $E^r$ primitives will be a subquotient of the module of the $E^1$ primitives.  It follows that $d^r$ can be nonzero only when $r = 2^t-2^s$ for some $0\leq s<t$.

We then determine universal differentials. After identifying $d^1$, we deduce how this propagates to give information about higher differentials.  We are able to do this by using the $\Z/2$ Tate construction to reveal how Dyer--Lashof operations are reflected in the spectral sequence associated to a tower with an `action' of an $E_{\infty}$ operad. Another consequence is that the module of $E^{\infty}$ primitives has a left action by Dyer--Lashof operations induced from the action on $E^1$, though, curiously, this is not true for the intervening pages.

Guided by our formula for universal differentials, we then construct, for $M \in \M$, an algebraic spectral sequence depending functorially on $M$.  The topological spectral sequence for $X$  agrees with the algebraic spectral sequence for $H_*(X;\Z/2)$ for many spectra $X$, including suspension spectra and almost all Eilenberg--MacLane spectra, and may be a subquotient in general. Our algebraic functor $E^{alg,\infty}_{*,*}(M)$ takes values in the category $\HQU$, and is built out of the derived functors of `destabilization' which were the subject of much research in the 1980's by W.~Singer \cite{singer}, J.~Lannes and S.~Zarati \cite{lz}, and P.~Goerss \cite{goerss}.

We now introduce our cast of characters, then describe our results in more detail.

\subsection{The tower for $\Sinfty_+ \Oinfty X$.}

We let $\T$ denote the category of based topological spaces and $\Sp$ the category of $S$--modules as in \cite{ekmm}.  The suspension spectrum functor $\Sinfty: \T \ra \Sp$ and the 0th space functor $\Oinfty: \Sp \ra \T$ induce an adjoint pair on homotopy categories.  We use the notation $\Sinfty_+ Z$ for the suspension spectrum of $Z_+$, the union of the space $Z$ with a disjoint basepoint.  $\Sinfty_+ Z$ comes with a natural augmentation map $\Sinfty_+ Z \ra \Sinfty_+ * = S$ to the sphere spectrum.

T.~Goodwillie's general theory of the calculus of functors \cite{goodwillie3}, applied to the endofunctor of $\Sp$ sending $X$ to $\Sinfty_+ \Oinfty X$, yields a natural tower $P(X)$ of fibrations augmented over $S$:
\begin{equation*}
\xymatrix{&&& \vdots \ar[d] \\
&&& P_3(X) \ar[d] \\
&&& P_2(X) \ar[d] \\
\Sinfty_+ \Oinfty X \ar[rrr]^-{e_1}  \ar[urrr]^-{e_2} \ar[uurrr]^-{e_3} &&& P_1(X).}
\end{equation*}

Basic properties include the following.
\begin{itemize}
\item $P_1(X)$ identifies with the product of spectra $X \times S$, and the map induced by  $e_1$ on fibers of the augmentation corresponds to the evaluation $ \Sinfty \Oinfty X \ra X$.
\item The fiber of the map $P_d(X) \ra P_{d-1}(X)$ is naturally equivalent to the spectrum $ D_{d}X = (X^{\sm d})_{h\Sigma_d}$, the $d$th extended power of $X$.
\item If $X$ is $0$--connected, then $e_j$ is $j$--connected.
\end{itemize}

We occasionally make use of the reduced tower $\tilde P(X)$ defined by letting $\tilde P_d(X)$ be the fiber of the augmentation $P_d(X) \ra S$.  There is a natural equivalence $P_d(X) \simeq \tilde P_d(X) \times S$ for all $d$.

\subsection{The spectral sequence}
Applying mod 2 homology to the tower $P(X)$ yields a left half plane spectral sequence $\{E^r_{*,*}(X)\}$.

\begin{prop} \label{SS prop} The spectral sequence satisfies the following properties. \\

\noindent{\bf (a)} \ It converges strongly  to $H_*(\Oinfty X)$ when $X$ is $0$--connected.  \\

\noindent{\bf (b)} \  $E^1_{-d,d+*}(X) = H_*(D_dX)$. \\

\noindent{\bf (c)} \ The Steenrod algebra $\A$ acts on the columns of the spectral sequence. \\

\noindent{\bf (d)} \ Dyer--Lashof operations $Q^i$, for all $i \in \Z$, act on $E^1_{*,*}(X)$, where they take the form
$Q^i: E^1_{-d, d+j} = H_j(D_dX) \ra H_{j+i}(D_{2d}X) = E^1_{-2d, 2d+i+j}$. \\

\noindent{\bf (e)} \  The spectral sequence is a spectral sequence of Hopf algebras.  The product and coproduct on $E^{\infty}$ are induced by the H-space product and diagonal on $\Oinfty X$, while the product and coproduct on $E^1$ are induced by the multiplication maps $D_bX \sm D_cX \ra D_{b+c}X$ and the transfer maps $D_{b+c}X \ra D_bX \sm D_cX$ associated to the subgroup inclusions $\Sigma_b \times \Sigma_c \hookrightarrow \Sigma_{b+c}$.
\end{prop}

The first three listed properties are evident from the properties of the tower.  Using G.~Arone's explicit model for this tower \cite{arone}, further properties were explored in \cite{ak}, including property (e).  There it was also shown that the action of the little cubes operad $\C_{\infty}$ on the infinite loop space $\Oinfty X$ induces a corresponding action on the tower.  This leads to the Dyer--Lashof operations of property (d).
How these operations correspond to the Dyer-Lashof operations on $H_*(\Oinfty X)$ at the level of $E^{\infty}_{*,*}(X)$, and how they act on the rest of the spectral sequence is part of the content of \thmref{theorem 1} below.

\subsection{Lots of categories and a global description of $E^1$}

We need to describe our spectral sequence $E^1$ term in a more global and functorial manner.
We assume some familiarity with Dyer--Lashof operations and Steenrod operations, but see \secref{homology section} for detailed references and a bit more detail.  Explicit formulae for the Nishida relations and the Dyer--Lashof Adem relations are given in \secref{top def section}.

We introduce various algebraic categories.

\begin{itemize}
\item $\M$ is the category of locally finite right $\A$--modules.  The Steenrod squares go down in degree: given $x \in M \in \M$, $|x Sq^i| = |x|-i$.  A right $\A$--module $M$ is locally finite if, for all $x \in M$, $x \cdot \A$ is finite dimensional over $\Z/2$.
\item
$\U$ is the full subcategory of $\M$ consisting of modules satisfying the unstable condition: $xSq^i = 0$ whenever $2i>|x|$.
\item $\Qq$ is the category of graded vector spaces $M$ acted on by Dyer--Lashof operations $Q^i: M_d \ra M_{d+i}$, for $i \in \Z$, satisfying the Adem relations and the unstable condition: $Q^ix = 0$ whenever $i<|x|$.
\item $\QM$ is the full subcategory of $\M \cap \Qq$ consisting of objects whose Dyer--Lashof structure is intertwined with the Steenrod structure via the Nishida relations.  \item $\QU = \QM \cap \U$.
\end{itemize}
All these categories are abelian, and admit tensor products, via the Cartan formula for both Steenrod and Dyer--Lashof operations.  Then we define various categories of Hopf algebras.
\begin{itemize}
\item $\HM$ is the category of bicommutative Hopf algebras in $\M$.
\item $\HQM$ is the category of bicommutative Hopf algebras in $\QM$ satisfying the Dyer--Lashof restriction axiom: $Q^{|x|}x = x^2$.  \item $\HQU = \HQM \cap \U$. \end{itemize}
We also need two `free' functors.
\begin{itemize}
\item  $\mathcal R_*: \M \ra \QM$ is left adjoint to the forgetful functor.  Explicitly,
$\mathcal R_*M = \bigoplus_{s=0}^{\infty} \mathcal R_s M$ where $\mathcal R_s: \M \ra \M$ is given by
$$\mathcal R_sM = \langle Q^Ix \ | \ l(I) = s,  x \in M\rangle/(\text{unstable and Adem relations}).$$
Here, if $I = (i_1,\dots,i_s)$, $Q^Ix = Q^{i_1}\cdots Q^{i_s}x$, and $l(I) = s$.
\item $U_{\Qq}: \QM \ra \HQM$ is left adjoint to the functor taking an object $H \in \HQM$ to its module $PH$ of primitives.  Explicitly,
    $$ U_{\Qq}(M) = S^*(M)/(Q^{|x|}x - x^2),$$
where $S^*(M)$ is the free commutative algebra generated by $M$, and the coalgebra structure is determined by making $M$ primitive.
\end{itemize}

The classic calculation of the mod 2 homology of extended powers can be interpreted as saying the following.

\begin{prop}  The natural maps $\mathcal R_s(H_*(X)) \ra H_*(D_{2^s}(X))$ are inclusions, and induce an isomorphism
$$  U_{\Qq}(\mathcal R_*(H_*(X))) = \bigoplus_{d=0}^{\infty} H_*(D_dX)$$
as  objects in $\HQM$.
\end{prop}

We begin our study of the spectral sequence with the following.

\begin{thm} \label{global ss thm}
\noindent{\bf (a)}\  $H_*(\Oinfty X)$ is an object in $\HQU$. \\

\noindent{\bf (b)} \  $ E^1_{*,*}(X) = U_{\Qq}(\mathcal R_*(H_*(X)))$
as an object in $\HQM$, and under this identification, the module of primitives $\mathcal R_s(H_*(X))$ is contained in $E^1_{-2^s,2^s+*}(X)$. \\

\noindent{\bf (c)} \ Steenrod operations act vertically, while Dyer--Lashof operations double the horizontal degree.\\

\noindent{\bf (d)} \ Each $E^r_{*,*}(X)$ is an object in $\HM$, and each $d^r$ is $\A$--linear and both a derivation and coderivation. \\

\noindent{\bf (e)} \ Each $E^r_{*,*}(X)$ is primitively generated, and its bigraded module of primitives $V^r_{*,*}(X)$ will be an $\A$--module subquotient of $\mathcal R_*(H_*(X))$.  \\

\noindent{\bf (f)} \ The only possible nonzero differentials will be $d^r$ when $r = 2^t-2^s$, with $t>s\geq 0$, and $d^{2^t-2^s}$ will be determined by its restriction to the primitives in $E^{2^t-2^s}_{-2^s,*}(X)$.
\end{thm}

The first four properties listed here are just restatements of parts of the last two propositions.  We will see that property (e) follows from these: it is standard that each $E^r$ will be primitively generated since this is true for $E^1$, but in our case we can also control where the $E^r$ primitives occur.  See \propref{Hopf alg prop} and the subsequent discussion.  Property (f) follows from (e), as $d^r$ will send primitives to primitives. \\

\begin{rem}  The careful reader may note that $H_*(\Oinfty X)$ satisfies one more condition than has been described: the dual of the classic restriction axiom for unstable $\A$--algebras, $Sq^{|x|}x = x^2$.  This property is not preserved by the filtration on $H_*(\Oinfty X)$.  The good news is then that this extra structure is available to be used to help determine extension problems.
\end{rem}

\subsection{Universal differentials}

We now identify universal structure on the spectral sequence.

\begin{thm} \label{theorem 1}  For all spectra $X$, the following hold in $\{E^r_{*,*}(X)\}$. \\

\noindent{\bf (a)} For all $x \in H_*(X)$, $\displaystyle d^1(x) = \sum_{i\geq 0} Q^{i-1}(xSq^{i})$.\\

\noindent{\bf (b)} If $y \in H_*(D_dX)$ lives to $E^r$, and $d^r(y)$ is represented by $z \in H_*(D_{d+r}X)$, then $Q^iy \in H_*(D_{2d}X)$ lives to $E^{2r}$, and $d^{2r}(Q^iy)$ is represented by $Q^i(z) \in H_*(D_{2d+2r}X)$.   \\

\noindent{\bf (c)} If $y \in H_*(D_dX)$ represents $z \in H_*(\Oinfty X)$ in $E^{\infty}_{-d,*}(X)$, then $Q^iy \in H_*(D_{2d}X)$ represents $Q^iz \in H_*(\Oinfty X)$ in $E^{\infty}_{-2d,*}(X)$.
\end{thm}

A consequence of the first two parts of the theorem is the following identification of universal differentials.

\begin{cor} \label{corollary 1}  For all spectra $X$, $x \in H_*(X)$, and $I$ of length $s$, $Q^Ix$ lives to $E^{2^s}_{-2^s,*}(X)$ and
$$d^{2^s}(Q^Ix) =  \sum_{i\geq 0} Q^I Q^{i-1}(xSq^{i}) \in E^{2^s}_{-2^{s+1},*}(X).$$
\end{cor}

To further give context to what \thmref{theorem 1} says about how Dyer--Lashof operations work in the spectral sequence, let
$$ \bold 0 = B^1 \subseteq B^2 \subseteq \dots \subseteq B^r \subseteq \dots \subseteq Z^r \subseteq \dots \subseteq Z^2 \subseteq Z^1 = E^1$$
be cycles and boundaries as usual, so that $E^r = Z^r/B^r$.  Then \thmref{theorem 1}(b) implies that for all $r$, Dyer--Lashof operations on $E^1$ restrict to maps
$$Q^i: Z^r \ra Z^{2r} \text{ \ and \ } Q^i: B^{r} \ra B^{2r-1}.$$  As $Z^{2r}/B^{2r-1}$ both includes into $E^{2r-1}$ and projects onto $E^{2r}$, one gets Dyer--Lashof operations of two flavors:
$$Q^i: E^r  \ra E^{2r-1} \text{ \ and \ } Q^i: E^r  \ra E^{2r}.$$
This discussion holds when $r = \infty$, leading to the next corollary.

\begin{cor} \label{corollary 2} For all spectra $X$, $E^{\infty}_{*,*}(X) \in \HQM$, with structure induced from $E^1$: $E^{\infty}_{*,*}(X) = U_{\Qq}(V^{\infty}_{*,*}(X))$, with the bigraded module of primitives $V^{\infty}_{*,*}(X) \in \QM$ equal to a subquotient of $\mathcal R_*(H_*(X))$.  To the extent that the spectral sequence converges, this structure is also induced from $H_*(\Oinfty X)$.
\end{cor}

\begin{rem}  Though both $E^1_{*,*}(X)$ and $E^{\infty}_{*,*}(X)$ admit Dyer-Lashof operations, $Q^i$ is {\em not} generally defined on the intervening pages, $E^r_{*,*}(X)$ with $1<r<\infty$, with the exception of the squaring operation $x \mapsto Q^{|x|}(x)$, which plays a special role in the Hopf algebra theory.  See \secref{CP2 example} for a simple example illustrating this.
\end{rem}

\begin{rem} Since $H_*(\Oinfty X)$ is always an unstable $\A$--module, it follows that if $X$ is 0--connected, then $E^{\infty}_{*,*}(X) \in \HQU$.  We wonder if this is the case for {\em all} spectra $X$.
\end{rem}

\thmref{theorem 1} will be proved in \secref{proof of thm 1}, supported by the results in the preceding background section.  We briefly comment on the proof.

Statement (a) amounts to a calculation of $\delta_*$, where $\delta: X \ra \Sigma D_2X$ is the connecting map of the cofibration sequence $D_2X \ra \tilde P_2X \ra \tilde P_1X \simeq X$.  When $|x|>0$, this was calculated (in dual formulation) by the first author in \cite{k1} by means of universal example, and it is not too hard to extend this to all $x$.

We give proofs of statements (b) and (c) that show that versions of these statements will hold in the spectral sequence associated to any tower of spectra admitting an action of the operad $\C_{\infty}$. The key idea is to use the (once desuspended) $\Z/2$ Tate construction in place of homotopy orbits. For example, in the course of the proof, one might hope to construct maps
$$  (P_d(X)^{\sm 2})_{h \Z/2} = D_2(P_d(X))\ra P_{2d}(X)$$
compatible with the standard maps $D_2(D_d(X)) \ra D_{2d}(X)$.  These don't exist, but we will  show that one {\em does} have maps
$$ t_{\Z/2}(P_d(X)^{\sm 2}) \ra P_{2d}(X)$$
which do the job, where $\displaystyle t_{\Z/2}(Y^{\sm 2}) \simeq \holim_{n} \Sigma^n D_2(\Sigma^{-n}Y)$, the  colinearization of $D_2(Y)$ in McCarthy's sense \cite{mccarthy}.  A technical point is that, at an appropriate moment, we need to pass from towers of $S$--modules to towers of $H\Z/2$--modules.

\subsection{An algebraic spectral sequence}

We now build an algebraic spectral sequence using only the differentials given by the formula in \corref{corollary 1}.  Our discovery is that this spectral sequence can be completely described, with an interesting $E^{\infty}$ term.

We need yet more terminology and notation related to the category $\U$.
\begin{itemize}
\item Let $\Oinfty: \M \ra \U$ be right adjoint to the inclusion.  Explicitly, $\Oinfty M$ is the largest unstable submodule of $M$.
\item  Let $\Omega: \U \ra \U$ be right adjoint to the suspension $\Sigma: \U \ra \U$. Explicitly, $\Omega M$ is the largest unstable submodule of $\Sigma^{-1}M$.
\item The functor $\Oinfty$ is left exact, and we let $\Oinfty_s: \M \ra \U$ denote the associated right derived functors.
\end{itemize}

It is convenient to let $L_sM = \Omega \Oinfty_s\Sigma^{1-s}M$.  (Note that $L_0M = \Oinfty M$.)  We observe that these functors to the category $\U$ have extra structure.

\begin{prop} \label{DL structure on derived functors prop} There are natural operations
$Q^i: L_sM \ra L_{s+1}M$ giving $L_*M$ the structure of an object in $\QU$; indeed, $L_*M$ is a natural subquotient of $\mathcal R_*M$, viewed as an object in $\QM$.
\end{prop}

\begin{thm} \label{theorem 2}  For all $M \in \M$, there is a left half plane spectral sequence $\{E^{alg, r}_{*,*}(M)\}$ described by the following properties. \\

\noindent{\bf (a)}  The spectral sequence is a functor of $M$ taking values in $\HM$, with Steenrod operations acting vertically, and with each $d^r$ both a derivation and coderivation. \\

\noindent{\bf (b)} $ E^{alg,1}_{*,*}(M) = U_{\Qq}(\mathcal R_*M)$ as an object in $\HQM$, with the module of primitives $\mathcal R_sM$ contained in $E^{alg,1}_{-2^s,2^s+*}(M)$. \\

\noindent{\bf (c)} $d^r$ is not zero only when $r=2^s$, and $d^{2^s}$ is determined by the formulae in \corref{corollary 1}: for $x \in M$ and $I$ of length $s$, $Q^Ix$ lives to
$E^{alg,2^s}_{-2^s,*}(M)$, and
$$d^{2^s}(Q^Ix) = \sum_{i \geq 0} Q^IQ^{i-1}(xSq^i).$$

\noindent{\bf (d)} For all $r$, $E^{alg,r}_{*,*}(M)$ is primitively generated with nonzero primitives concentrated in the $-2^s$ lines.  For all $r>2^s$, the module of primitives in $E^{alg,r}_{-2^s,2^s+ *}(M)$ is naturally isomorphic to $L_sM$. \\

\noindent{\bf (e)}
$ E^{alg, \infty}_{*,*}(M) \simeq U_{\Qq}(L_*M)$ as an object in $\HQU$.
\end{thm}

For a spectrum $X$, the spectral sequences $\{E^r_{*,*}(X)\}$ and $\{E^{alg, r}_{*,*}(H_*(X))\}$
will agree exactly when all topological spectral sequence differentials $d^{2^t-2^s}$ with $s<t-1$ are zero.  We call a nonzero differential $d^{2^t-2^s}$ with $s<t-1$ a {\em rogue} differential. \\

\begin{cor} \label{alg = top cor}  If $d^{2^t-2^s}$ is a rogue differential with $2^t-2^s$ smallest, then restricted to the primitives on the $-2^s$ line, it will be a nonzero map
$$ L_sH_*(X) \ra E^{2^t-2^s}_{-2^{t},2^t + *-1}(X).$$
\end{cor}

Recalling that $L_0H_*(X) = \Oinfty H_*(X)$, it follows that rogue differentials off of the $-1$ line measure the failure of the evaluation map $H_*(\Oinfty X) \ra \Oinfty H_*(X)$ to be onto.  \thmref{theorem 1}(b) and the last corollary then tell us that in some circumstances this can be the only source of rogue differentials.

\begin{cor} \label{alg = top cor 2} Suppose the following conditions hold for a spectrum $X$. \\

\noindent{\bf (a)} \ $\Oinfty H_*(X) = L_0H_*(X)$ generates $L_*H_*(X)$ as a module over the Dyer--Lashof algebra. \\

\noindent{\bf (b)} \ The evaluation map $H_*(\Oinfty X) \ra \Oinfty H_*(X)$ is onto. \\

\noindent Then  $\{E^r_{*,*}(X)\} = \{E^{alg,r}_{*,*}(H_*(X))\}$, and thus $ E^{\infty}_{*,*}(X) \simeq U_{\Qq}(L_*H_*(X)))$ as an object in $\HQU$.
\end{cor}

\begin{rem}  One might wonder if $E^{\infty}_{*,*}(X)$ is always a subquotient of  $E^{alg,\infty}_{*,*}(H_*(X))= U_{\Qq}(L_*H_*(X)))$.  Our results say that the algebraic boundaries, $B^{alg,\infty}$, are contained in the topological boundaries $B^{\infty}$.  To conclude that $E^{\infty}$ is a subquotient of $E^{alg,\infty}$, it suffices to show that $Z^{\infty} \subseteq B^{\infty} + Z^{alg,\infty}$, when restricted to primitives.
\end{rem}

The development of our algebraic spectral sequence, and the proof of its properties as in \thmref{theorem 2}, is given in \secref{alg spec seq section}.  Besides using the Hopf algebra theory needed in the proof of \thmref{global ss thm}(e), this relies heavily on \secref{derived functors section}, which is focused on the connection between $\mathcal R_s$ and $\Oinfty_s$.  We say a bit about this connection here.

We relabel: let $R_s = \Sigma\mathcal R_s \Sigma^{s-1}: \M \ra \M$.  Let $d_s: R_s M \ra R_{s+1}M$ be given by the formula
$$ d_s(Q^Ix) = \sum_{i\geq 0} Q^IQ^{i-1}(xSq^{i}),$$
where we have suppressed some suspensions.  The next theorem is a variant of theorems in \cite{goerss} and \cite{powell}.  All such results go back to work of Singer \cite{singer} for inspiration.

\begin{thm} \label{Rs thm} For all $M \in \M$,
$$ R_0M \xra{d_0} R_1M \xra{d_1} R_2M \xra{d_2} \cdots$$
is a chain complex with $H_s(R_*M;d)$ naturally isomorphic to  $\Oinfty_sM$.
\end{thm}

\begin{rem} Recall that $L_sM = \Omega \Oinfty_s(\Sigma^{1-s}M)$. \thmref{Rs thm} says that
$$ \mathcal R_{s-1}(\Sigma^{-1} M) \xra{d_{s-1}} \mathcal R_s(M) \xra{d_s} \mathcal R_{s+1}(\Sigma M)$$
has homology $\Sigma^{-1}\Oinfty_s(\Sigma^{1-s}M)$ at $\mathcal R_s(M)$, and we will see that $L_sM$ is the image of a natural homomorphism
$$ \Oinfty_s(\Sigma^{-s}M) \ra \Sigma^{-1}\Oinfty_s(\Sigma^{1-s}M).$$
This may make it plausible (though by no means obvious) that there might be an algebraic spectral sequence with $E^1 = U_{\Qq}(\mathcal R_*M)$ and $E^{\infty} = U_{\Qq}(L_*M)$.
\end{rem}

In \secref{derived functors section}, we will give a complete presentation of \thmref{Rs thm}, which is much more topologically based and less explicitly computational than similar results in the literature. Also included in this section is a proof of \propref{DL structure on derived functors prop}.

\subsection{Examples} \label{example subsection}  In \secref{example section}, we give a variety of examples illustrating the strength of our main results and their limitations.  Here we summarize some of our findings.

\corref{alg = top cor 2} holds for the following families of spectra $X$, so that the algebraic and topological spectral sequences agree, and thus
$$E^{\infty}_{*,*}(X) \simeq U_{\Qq}(L_*H_*(X)).$$
\begin{itemize}
\item $X=\Sigma^n HA$, the Eilenberg--MacLane spectrum of type $(A,n)$, unless $n=0$ or $-1$, and $A$ has 2-torsion of order at least 4.
\item $X= \Sinfty Z$, a suspension spectrum.
\item $X = S^1\langle 1 \rangle$, the 1--connected cover of $\Sinfty S^1$.
\item $X = (\Sinfty Z)(s)$, the $s^{th}$ stage of an Adams resolution of $\Sinfty Z$, with $Z$ a connected space.
\end{itemize}

Verifying the geometric hypothesis of \corref{alg = top cor 2} for this last family relies on unpublished work of Lannes and Zarati \cite{lz2} from the 1980's.

By contrast, we have examples of spectra for which the spectral sequence has rogue differentials.
\begin{itemize}
\item When $X = \hocofib \{ \Sinfty \R P^4 \xra{4} \Sinfty \R P^4\}$, $d^3$ is nonzero.
\item When $X = H\Z/2^r$, with $r \geq 2$, $d^{2^r-1}$ is the only rogue differential.
\item When $X = \Sigma^{-1}H\Z/2^r$, with $r \geq 2$, the rogue differentials are the family  $d^{2^s(2^r-1)}$, $s \geq 0$.
\end{itemize}

Note that, with $X$ as in the first of these examples, $X$ and  the suspension spectrum of $\R P^4 \vee \Sigma \R P^4$ have isomorphic homology, but differing spectral sequences.

In all of these examples, $L_0H_*(X)$ turns out to generate $L_*H_*(X)$ as a module over the Dyer--Lashof algebra.  For a simple example where this does {\em not} hold, one can let $X = S^1 \cup_{\eta}D^3$: see \secref{CP2 example} for details.

\begin{rem} From our calculations, we learn that, if $X$ is {\em any} Eilenberg--MacLane spectrum whose homotopy is a finite abelian 2-group, the topological spectral sequence correctly computes $H_*(\Oinfty X)$, even when $X$ is not 0--connected.  In ongoing work, the second author has found a couple more examples of nonconnective spectra for which the spectral sequence correctly computes $H_*(\Oinfty X)$.  As of yet the authors have no good sense of when such exotic convergence should be expected.
\end{rem}

\subsection{Acknowledgments} This research was partially supported by National Science Foundation grant 0967649.  Some of these results were presented by the first author at the September, 2011 Workshop on Homotopy Theory at Oberwolfach, with a report published as \cite{kuhnOber11}, and by the second author in the Special Session on the Calculus of Functors at the January 2012 AMS Annual Meeting in Boston.  The authors would like to thank the referee for many constructive comments and, in particular, for nudging us towards a much improved \secref{alg spec seq section}.

\section{Preliminaries} \label{geometric preliminaries}

\subsection{Prerequisites on spectra}

$\T$ will be the category of pointed topological spaces, and $\Sp$ the category of $S$--modules as in \cite{ekmm}.  An $S$--module $X$ is a spectrum of the classic sort (as in \cite{lmms}) equipped with extra structure, and we let $X_n$ denote its $n$th space.  Thus $\Oinfty X = X_0$.

By a weak natural transformation $F \ra G$ between two functors with values in a model category, we mean a zig-zig of natural transformations $F \la H \ra G$ (or $F \ra H \la G$) for which the backwards arrow is a weak equivalence (on any object).  We say that a diagram of such weak natural transformations commutes if it induces a commutative diagram in the homotopy category (on each object).

Though we will try to not dwell too deeply on the details of the model, studied in \cite{ak}, for our Goodwillie tower, the following proposition summarizes the formal properties of $S$--modules that are needed to make the arguments in \cite{ak} work.

\begin{prop} \label{S-mod prop} The category $\Sp$ of $S$--modules has the following structure.

\begin{itemize}
\item  $\Sp$ is a category enriched over $\T$.
\item $\Sp$ is tensored and cotensored over $\T$: given $K \in \T$ and $X \in \Sp$, there are spectra $K \sm X$ and $\Map_{\Sp}(K,X)$, natural in both variables, satisfying standard adjunction properties.
\item There are natural maps $\eta: \Map_{\Sp}(K,X) \ra \Map_{\Sp}(L \sm K, L \sm X)$.
\item There are natural maps $$\Map_{\Sp}(K,X) \sm \Map_{\Sp}(L,Y) \ra \Map_{\Sp}(K \sm L, X \sm Y),$$ which are weak equivalences if $K$ and $L$ are finite CW complexes.
\item The suspension spectrum functor $\Sinfty: \T \ra \Sp$ commutes with smash product.
\item There are natural maps $e: \Sinfty \Map_{\T}(K,Z) \ra \Map_{\Sp}(K, \Sinfty Z)$.
\item There is a weak natural equivalence $\displaystyle \hocolim_n \Sigma^{-n}\Sinfty X_n \ra X$.
\end{itemize}

\end{prop}

Here and elsewhere we write $\Sigma^{-n}X$ for $\Map_{\Sp}(S^n,X)$.

It is only the last item that really needs comment.  See Appendix \ref{spectra appendix} for more discussion of this point.

We end this subsection by describing the setting for the `evaluation/diagonal' natural transformations
$$ \epsilon: \Sigma F(X) \ra F(\Sigma X)$$
which play a significant role in our work.

As $\Sp$ is a category enriched over $\T$, $\Mor_{\Sp}(X,Y)$ has the structure of a based topological space.  A functor $F: \Sp \ra \Sp$ is said to be {\em continuous} if
$$ F: \Mor_{\Sp}(X,Y) \ra \Mor_{\Sp}(F(X),F(Y))$$
is a continuous function.  If $F$ is also {\em reduced}, i.e.\ $F(*)=*$, then this continuous function is also based.

\begin{defn}  Given a continuous reduced functor $F: \Sp \ra \Sp$, and $K \in \T$, we let
$$ \epsilon: K \sm F(X) \ra F(K \sm X)$$
be adjoint to the composite of continuous functions
$$ K \ra \Mor_{\Sp}(X, K \sm X) \xra{F} \Mor_{\Sp}(F(X), F(K \sm X)),$$
where the first map is the unit of the adjunction
$$ \Mor_{\Sp}(K \sm X, Y) \simeq \Mor_{\T}(K, \Mor_{\Sp}(X,Y)).$$
\end{defn}

\subsection{The Tate construction} If $G$ is a finite group, we let $G$--$\Sp$ denote the category of $S$--modules with a $G$--action: the category of `naive' $G$--spectra.

More generally, if $R$ is a commutative $S$--algebra, we let $G$--$R$--mod be the category of $R$--modules with $G$--action.  (For us, $R$ will eventually be $H\Z/2$.)

Given $Y \in G$--$R$--mod, we let $Y_{hG}$ and $Y^{hG}$ respectively denote associated homotopy orbit and homotopy fixed point $R$--modules.

The homotopy orbit construction satisfies a change-of-rings lemma.

\begin{lem} Given $Y \in G$--$\Sp$ and a commutative $S$--algebra $R$, there is a natural isomorphism of $R$--modules, $R \sm Y_{hG} = (R \sm Y)_{hG}$.
\end{lem}

There are various constructions in the literature, e.g.\ \cite{acd, ak, gm}, of a natural norm map
$$ N_G(Y): Y_{hG} \ra Y^{hG}.$$
The Tate spectrum of $Y$ is defined as the homotopy cofiber of $N_G(Y)$.  It will be more convenient for us to desuspend this once and define $t_G(Y)$ to be the homotopy {\em fiber} of $N_G(Y)$.  Thus $t_G(Y)$ comes equipped with a natural transformation $t_G(Y) \ra Y_{hG}$.

The next lemma lists the properties we need about this.

\begin{lem} {\bf (a)} $t_G$ takes weak equivalences and cofibration sequences in $G$--$R$--mod to weak equivalences and cofibration sequences in $R$--mod. \\

\noindent{\bf (b)} If $X$ is a nonequivariant $R$--module, $t_G(G_+ \sm X) \simeq *$.
\end{lem}

See \cite[Part~I]{gm} for these sorts of facts.  Statement (b) also follows from \cite[Prop.2.10]{ak}.

When $G=\Z/2$, there is a well known model for $t_{\Z/2}(Y)$.  Let $\rho$ be the one dimensional real sign representation of $\Z/2$, and let $S^{n\rho}$ be the one point compactification of $n\rho$.

\begin{lem} (Compare with \cite[Thm.16.1]{gm}.)  For $Y \in \Z/2$--$R$--mod, there is a natural weak equivalence
$$ t_{\Z/2}(Y) \simeq \holim_n \Map_{R\text{-mod}}(S^{n\rho}, Y)_{h\Z/2}.$$
\end{lem}

We now specialize to the special case $Y = X\sm_R X$, with $X$ an $R$--module.

\begin{notation} Let $X$ be an $R$--module. We let $D^R_2(X) = (X\sm_R X)_{h\Z/2}$ and $\D^R_2(X) = t_{\Z/2}(X\sm_R X)$.
\end{notation}

One easily checks the following.

\begin{lem} {\bf (a)} For all $S$--modules $X$ and commutative $S$--algebras $R$, there is an isomorphism of $R$--modules, $D^R_2(R \sm X) = R \sm D_2(X)$. \\

\noindent{\bf (b)} For all $R$--modules $X$, there is a natural weak equivalence
$$\D^R_2(X) \simeq \holim_n \Sigma^n D^R_2(\Sigma^{-n}X).$$
\end{lem}

Thus $\D^R_2$ is identified as the colinearization of $D^R_2$ in the sense of \cite{mccarthy}.

\begin{cor} \label{linear cor} $\D^R_2$ preserves cofibration sequences of $R$--modules.
\end{cor}

\subsection{The homology of extended powers} \label{homology section} When $X$ is a spectrum, a construction of the Dyer--Lashof operations
$$Q^i: H_j(D_dX) \ra H_{j+i}(D_{2d}X),$$
for all $i \in \Z$, is given by M.~Steinberger in \cite[Thm.III.1.1]{bmms}.  An alternative construction is given later in the same book by J.~McClure \cite[Prop.VIII.3.3]{bmms}.  He \cite[Thm.IX.2.1]{bmms} also computes
$H_*(\mathbb P X)$
as an algebra with both Dyer--Lashof and Steenrod operations, where $\mathbb P X = \bigvee_{d=0}^{\infty} D_dX$.

The coproduct structure on $H_*(\mathbb P X)$ seems to be less well documented in the literature. Recall that the coproduct $\Delta$ is induced by the transfer maps
$ t_{b,c}: D_{b+c}X \ra D_bX \sm D_cX$.  The following lemma is presumably well known, and is analogous to \cite[Thm.I.1.1(6)]{clm}.

\begin{lem} For all $y \in H_*(\mathbb P X)$, if $\Delta(y) = \sum y^{\prime} \otimes y^{\prime \prime}$, then
$$\Delta(Q^ky) = \sum_{i+j = k} \sum Q^iy^{\prime} \otimes Q^jy^{\prime \prime}.$$
\end{lem}
\begin{proof}[Sketch proof] Let $p: X \ra X \vee X$ be the pinch map.  If $b+c = d$, then $t_{b,c}$ is the $(b,c)$th component of the composite
$$ D_dX \xra{D_d(p)}D_d(X \vee X) = \bigvee_{b+c=d} D_b X \sm D_cX.$$
The diagram
\begin{equation*}
\xymatrix{
D_2D_dX \ar[d] \ar[rr]^-{D_2D_d(p)} && D_2D_d(X\vee X)\ar[d]  \\
D_{2d}X \ar[rr]^-{D_{2d}(p)} && D_{2d}(X\vee X) }
\end{equation*}
commutes, and the lemma follows from this, using the Cartan formula $Q^k(y^{\prime} \otimes y^{\prime \prime}) = \sum_{i+j=k}Q^iy^{\prime} \otimes Q^jy^{\prime \prime}$.
\end{proof}

Crucial to us is the behavior of $\epsilon: \Sigma D_d X \ra D_d \Sigma X$ on homology.

\begin{lem} \label{epsilon lemma} {\bf (a)} \ $\epsilon_*: H_*(\mathbb P X) \ra H_{*+1}(\mathbb P \Sigma X)$ sends the algebra decomposables to zero, and has image in the coalgebra primitives. \\

\noindent{\bf (b)} \ $\epsilon_*(Q^iy) = Q^i(\epsilon_*(y))$.
\end{lem}

One reference for (a) is \cite[Ex.6.7]{ak}.  For statement (b), see \cite[Lem.II.5.6]{bmms} (or alternatively, deduce it from \cite[Prop.VIII.3.2]{bmms}).

\begin{cor}  The image of $\epsilon_*: H_*(\Sigma D_{2^s}(\Sigma^{-1} X)) \ra H_*(D_{2^s}X)$ is precisely the subspace of primitives: the span of the elements $Q^Ix$ with $l(I)=s$ and $x \in H_*(X)$.
\end{cor}

\subsection{Dyer--Lashof operations for $\D_2$}

\begin{lem} {\bf (a)} \ The sequence
$$ \dots \xra{\epsilon_*} H_{*-2}(D_2(\Sigma^{-2}X))\xra{\epsilon_*} H_{*-1}(D_2(\Sigma^{-1}X))\xra{\epsilon_*} H_*(D_2(X))$$
is Mittag--Leffler. \\

\noindent{\bf (b)} \ $\displaystyle \pi_*(\D^{H\Z/2}_2(H\Z/2 \sm X)) = \lim_n H_{*}(\Sigma^n D_2(\Sigma^{-n}X))$.
\end{lem}

Statement (a) follows from \lemref{epsilon lemma}, and then (b) follows from (a), noting that $\pi_*(D_2^{H\Z/2}(H\Z/2 \sm X)) = H_*(D_2(X))$.

\begin{cor} \label{Tate Dyer Lashof cor}
The natural transformation
$$ Q^i: H_*(X) \ra H_{*+i}(D_{2}(X))$$
lifts to a natural transformation
$$ Q^i: H_*(X) \ra \pi_{*+i}(\D_{2}^{H\Z/2}(H\Z/2 \sm X)).$$
\end{cor}

This corollary will play a critical role in our proof of parts (b) and (c) of \thmref{theorem 1}: see \secref{strategy section}.

\subsection{The cohomology of $D_2X$}  \label{cohomology section} In the proof of \thmref{theorem 1}(a), it will be useful to work with mod 2 cohomology.  As in \cite{k1}, let
$$ \hat Q_0: H^*(X) \ra H^{2*}(D_2X)$$
be the squaring operation, and then, for $i>0$, let
$$ \hat Q_i: H^*(X) \ra H^{2*+i}(D_2X)$$
be defined to be the composite
$$ H^*(X) = H^{*+i}(\Sigma^{i} X) \xra{\hat Q_0} H^{2*+2i}(D_2(\Sigma^i X)) \xra{\epsilon^*} H^{2*+i}(D_2 X).$$
One also has a product $*: H^*(X) \otimes H^*(X) \ra H^*(D_2X)$ induced by $t_{1,1}: D_2X \ra X\sm X$.  One has $\hat Q_0(x+y) = \hat Q_0x + \hat Q_0y + x*y$, while, for $i>0$, $\hat Q_i$ is linear.

\begin{lem} $H^*(D_2X)$ is spanned by the elements $\hat Q_ix$ and $x*y$.
\end{lem}

These operations are appropriately dual to the homology Dyer--Lashof operations.  In the next proposition, $Q_ix = Q^{i+|x|}x$, as is standard.

\begin{prop} \label{dual prop}  Let $\langle x,y \rangle$ denote the cohomology/homology pairing.  Given $w,x \in H^*(X)$ and $y,z \in H_*(X)$, the following formulae hold.\\

\noindent{\bf (a)} \ $
\langle {\hat Q_i x},{Q_j y} \rangle =
\begin{cases}
\langle {x},{y} \rangle & \text{if } i=j \\ 0 & otherwise.
\end{cases}
$\\

\noindent{\bf (b)} \ $
\langle {\hat Q_i x},{y*z} \rangle =
\begin{cases}
\langle {x},{y} \rangle \langle {x},{z} \rangle& \text{if } i = 0 \\ 0 & otherwise.
\end{cases}
$\\

\noindent{\bf (c)} \ $
\langle {w*x},{Q_iy} \rangle =
\begin{cases}
\langle {w},{y} \rangle \langle {x},{y} \rangle & \text{if } i = 0 \\ 0 & otherwise.
\end{cases}
$\\

\noindent{\bf (d)} \ $
\langle {w*x},{y*z} \rangle =
\langle {w},{y} \rangle \langle {x},{z} \rangle + \langle {w},{z} \rangle \langle {x},{y} \rangle.
$
\end{prop}

See \cite[Prop.A.1]{k1}.

\section{Proof of \thmref{theorem 1}} \label{proof of thm 1}

\subsection{Proof of \thmref{theorem 1}(a)} It suffices to prove this formula assuming $X$ is a spectrum whose homology is bounded below and of finite type.  In this case, it is easiest to first prove the cohomology version of \thmref{theorem 1}(a).

Recall from \secref{cohomology section} that $H^*(D_2 X)$ is spanned by elements $\hat Q_i x$ and $x*y$, with $x,y \in H^*(X)$ and $i \geq 0$.

As in the introduction, let $\delta: X \ra \Sigma D_2 X$ be the connecting map of the cofibration sequence $D_2 X \ra \tilde{P}_2 X \ra X$.

\begin{prop}
For $x \in H^n(X)$, we have $\delta^*(\sigma \hat{Q}_rx) = Sq^{r + n + 1} x$.
\end{prop}
\begin{proof}
The proof uses ideas from \cite[Prop.4.3]{k1} and \cite[Appendix~A]{kuhntate}.

Let $P(r, n)$ be the statement $$\delta^*(\sigma \hat{Q}_rx) = Sq^{r + n + 1} x \text{ for all } x \in H^n(X).$$
We need to prove that $P(r,n)$ is true for all $r \geq 0$ and $n \in \Z$.

We first observe that, for $r>0$, $P(r-1,n+1)$ implies $P(r,n)$. To see this, we use that the diagram
$$
\xymatrix{
  X \ar[r]^-{\delta} & \Sigma D_2 X \ar[d]^{\epsilon} \\
  X \ar@{=}[u] \ar[r]_-{\Sigma^{-1} \delta} & D_2 \Sigma X}
$$
commutes by the naturality of $\delta$.
So, if $x \in H^n(X)$ and $P(r-1, n+1)$ holds, then
\begin{equation*}
\begin{split}
\delta^*(\sigma \hat{Q}_rx)
  &= \delta^*(\epsilon^*(\hat{Q}_{r-1} \sigma x)) \\
  &= (\Sigma^{-1} \delta)^*(\hat{Q}_{r-1}\sigma x) \\
  &= \sigma^{-1} Sq^{(r-1) + (n+1) + 1} \sigma x \\
  &= Sq^{r + n + 1} x.
\end{split}
\end{equation*}

Thus it suffices to show $P(0, n)$ for all $n$.  By naturality, it is enough to show that
$$\delta^*(\sigma \hat{Q}_0 \iota_n) = Sq^{n + 1} \iota_n,$$
where $\iota_n \in H^n(\Sigma^n H\Z/2)$ is the fundamental class.

We break this into cases.

When $n > 0$, this was proven in \cite[Prop.4.3]{k1} as follows. As $\Sigma^n H\Z/2$ is 0--connected, the cohomology tower spectral sequence for $\Sigma^n H\Z/2$ strongly converges to $H^*(K(\Z/2, n))$.  Thus the element $Sq^{n+1} \iota_n$ must be an eventual boundary, as $Sq^{n+1} \iota_n = 0$ in $H^*(K(\Z/2, n))$.
For degree reasons, the only way this could happen is if $\delta^*(\sigma \hat{Q}_0\iota_n) = Sq^{n+1} \iota_n$.

When $n < -1$, the degree of $\sigma \hat{Q}_0\iota_n$ is $2n + 1 < n$, so $\delta^*$ takes this element to zero.  As desired, $Sq^{n + 1} \iota_n$ is also zero since $n + 1 < 0$.

For the remaining cases, we use the fact that $Sq^2$ is injective on $\A$ in degrees 0 and 1.

We have $Sq^2 \delta^*(\sigma \hat{Q}_0\iota_n) = \delta^*(\sigma Sq^2 \hat{Q}_0\iota_n)$.
The Nishida relations for the operation $\hat Q_0$ \cite[Prop.3.15]{k1} tell us that
$$Sq^2 \hat{Q}_0\iota_n = \binom{n}{2} \hat{Q}_2\iota_n + \binom{n-1}{0} \hat{Q}_0Sq^1 \iota_n +  \iota_n * Sq^2 \iota_n.$$
Since $\delta^*$ takes nontrivial products to zero, we deduce that
$$Sq^2 \delta^*(\sigma \hat{Q}_0\iota_n) = \binom{n}{2}\delta^*(\sigma \hat{Q}_2\iota_n) + \delta^*(\sigma \hat{Q}_0Sq^1 \iota_n).$$

When $n = 0$, this equation and the established fact $P(0, 1)$ imply that
$$Sq^2 \delta^*(\sigma \hat{Q}_0\iota_0) = Sq^2 Sq^1 \iota_0.$$
We deduce that $\delta^*(\sigma \hat{Q}_0\iota_0) = Sq^1 \iota_0$.

When $n = -1$, the above equation and the established facts $P(2, -1)$ (implied by $P(0, 1)$) and $P(0, 0)$ imply that $$Sq^2 \delta^*(\sigma \hat{Q}_0\iota_{-1}) = Sq^2 \iota_{-1} + Sq^1 Sq^1 \iota_{-1} = Sq^2 \iota_{-1}.$$
We deduce that $\delta^*(\sigma \hat{Q}_0\iota_{-1}) = \iota_{-1}$.
\end{proof}

If we define $\hat{Q}^ix = \hat{Q}_{i-|x|}x$, the proposition says that, for all $x \in H^*(X)$,
$$\delta^*(\sigma \hat{Q}^ix) = Sq^{i + 1} x.$$
By duality, we get the formula stated in \thmref{theorem 1}(a): for all $x \in H_*(X)$,
$$ \delta_*(x) = \sum_{i \geq 0} \sigma Q^{i-1}(x Sq^{i}),$$

\begin{rem} Variants of the formula in \thmref{theorem 1}(a) go back at least as far as the 1966 paper \cite{six author}.
\end{rem}

\subsection{The strategy for the proof of statements (b) and (c)} \label{strategy section}

We outline the strategy of the proof of \thmref{theorem 1}(b) and (c).

First of all, what do we have to show?

In (c), the statement
$$ \text{`$y \in H_*(D_dX)$ represents $z \in H_*(\Oinfty X)$'} $$
means that, under the maps
$$ \Sinfty_+ \Oinfty X \xra{e_d} P_d(X) \xla{i_d} D_d(X),$$
we have $e_{d*}(z) = i_{d*}(y)$.  So to prove (c), we just need to show that then,
under the maps
$$ \Sinfty_+ \Oinfty X \xra{e_{2d}} P_{2d}(X) \xla{i_{2d}} D_{2d}(X),$$
we have $e_{2d*}(Q^iz) = i_{2d*}(Q^iy)$.

In (b), the statement
$$ \text{`$y \in H_*(D_dX)$ lives to $E^r$, and $d^r(y)$ is represented by $z \in H_*(D_{d+r}X)$'}$$
means that there is an element $w \in H_*(P_{d+r-1}(X))$, such that under the maps
$$ D_d(X) \xra{i_d} P_d(X) \xla{p_{d+r-1,d}} P_{d+r-1}(X) \xra{\delta_{d+r-1}} \Sigma D_{d+r}(X),$$
we have $i_{d*}(y) = {p_{d+r-1,d}}_*(w)$ and ${\delta_{d+r-1}}_*(w) = \sigma z$.  So to prove (b), we just need to show that then, there is an element $w_i \in H_*(P_{2d+2r-1}(X))$ such that
under the maps
$$ D_{2d}(X) \xra{i_{2d}} P_{2d}(X) \xla{p_{2d+2r-1,2d}} P_{2d+2r-1}(X) \xra{\delta_{2d+2r-1}} \Sigma D_{2d+2r}(X),$$
we have $i_{2d*}(Q^iy) = {p_{2d+2r-1,d}}_*(w_i)$ and ${\delta_{2d+2r-1}}_*(w_i) = \sigma Q^iz$. \\

We can also pass to $H\Z/2$--modules, and use elements in homotopy.  {\em Suppressing this from our notation, we will assume this in what we do below.}  For example, $H_*(\D_2(P_d(X)))$ will `really' mean $\pi_*(\D_2^{H\Z/2}(H\Z/2 \sm P_d(X)))$. \\

\thmref{theorem 1}(c) follows from the following.

\begin{prop} \label{thm1c prop} There is a commutative diagram of weak natural transformations
\begin{equation*}
\xymatrix{
\D_2(\Sinfty_+ \Oinfty X) \ar[r]^-{\D_2e_d} \ar[d]& \D_2(P_{d}(X)) \ar[d]&  \D_2(D_{d}(X)) \ar[l]_{\D_2i_d} \ar[d]  \\
\Sinfty_+ \Oinfty X \ar[r]^-{e_{2d}} & P_{2d}(X) &  D_{2d}(X) \ar[l]_{i_{2d}}}
\end{equation*}
in which left and right vertical maps are the composites
$$ \D_2(\Sinfty_+ \Oinfty X) \ra D_2(\Sinfty_+ \Oinfty X) \xra{\mu} \Sinfty_+ \Oinfty X$$
and
$$ \D_2(D_d(X)) \ra D_2(D_d(X)) \xra{\mu} D_{2d}(X),$$
where $\mu$ is the standard operad action.
\end{prop}

To deduce \thmref{theorem 1}(c) from this, suppose that $e_{d*}(z) = i_{d*}(y)$ as in the discussion above.  Then the diagram shows that $e_{2d*}(Q^iz) = i_{2d*}(Q^iy)$, where $Q^iz$ and $Q^iy$ can be viewed as being in the homology of the appropriate Tate construction, courtesy of \corref{Tate Dyer Lashof cor}.

\thmref{theorem 1}(b) follows from the following.

\begin{prop} \label{thm1b prop} There is a commutative diagram of weak natural transformations
\begin{equation*}
\xymatrix{
\D_{2}D_d(X) \ar[r]^-{\D_2i} \ar[d]& \D_2 P_{d}(X) \ar[d]&  \D_2 P_{d+r-1}(X) \ar[r]^-{\D_2\delta} \ar[l]_-{\D_2p} \ar[d]& \Sigma \D_2D_{d+r}(X) \ar[d]  \\
D_{2d}(X) \ar[r]^-{i} & P_{2d}(X) &  P_{2d+2r-1}(X) \ar[r]^-{\delta} \ar[l]_-{p}& \Sigma D_{2d+2r}(X)}
\end{equation*}
in which left and right vertical maps are as in the previous proposition.
\end{prop}
In interpreting the right square in this diagram, one should recall that $\Sigma \D_2D_{d+r}(X) \simeq \D_2\Sigma D_{d+r}(X)$, thanks to \corref{linear cor}.

To deduce \thmref{theorem 1}(b) from this, the needed element $w_i \in H_*(P_{2d+2r-1}(X))$ will then be the image of $Q^iw \in H_*(\D_2 P_{d+r-1}(X))$ under the vertical map second from the right.

It remains to prove these two propositions.  We do this at the end of the next subsection.

\subsection{Operad actions on towers}

The following definition is from \cite{ak}.
\begin{defn}  If $P$ is a tower in $\Sp$, then $P\sm P$ is the tower in $\Z/2$--$\Sp$ with
$$ (P \sm P)_d = \holim_{b+c \leq d} P_b \sm P_c.$$
\end{defn}

Suggestively, we will let $D_d$ denote the fiber of $P_d \ra P_{d-1}$, and then let $F_d$ denote the fiber of $(P\sm P)_d \ra (P\sm P)_{d-1}$.  From \cite[Cor.5.3]{ak} we learn

\begin{lem} There is a weak natural equivalence in $\Z/2$--$\Sp$
$$ F_d \simeq \prod_{b+c =d} D_b \sm D_c.$$
\end{lem}

Note that there are $\Z/2$--equivariant maps
$$ (P \sm P)_{2d+1} \ra (P \sm P)_{2d} \ra P_d \sm P_d$$
and
$$ F_{2d} \ra D_d \sm D_d.$$

\begin{lem} \label{Tate vanishing stuff}  These maps induce equivalences of Tate spectra:
$$ t_{\Z/2}((P \sm P)_{2d+1}) \xra{\sim} t_{\Z/2}((P \sm P)_{2d}) \xra{\sim} \D_2(P_d)$$
and
$$  t_{\Z/2}(F_{2d}) \xra{\sim} \D_2(D_d).$$
\end{lem}
\begin{proof} With $\epsilon$ either 0 or 1, filtered in the usual way, $(P \sm P)_{2d+\epsilon}$ has composition factors of two types:
\begin{itemize}
\item $D_i \sm D_i$ with $i \leq d$.
\item $\Z/2_+ \sm D_i \sm D_j$ with $i<j$ and $i+j \leq 2d+\epsilon$.
\end{itemize}
Meanwhile
$P_d \sm P_d$ has composition factors:
\begin{itemize}
\item $D_i \sm D_i$ with $i \leq d$.
\item $\Z/2_+ \sm D_i \sm D_j$ with $i <j \leq d$.
\end{itemize}
The first type of factors match up, and after applying $t_{\Z/2}$, the second type become null.

The proof for $F_{2d}$ is similar and easier.
\end{proof}

Now let $P$ be the tower $P(X)$, the Goodwillie tower for $\Sinfty_+ \Oinfty X$.

Recall that the $\C_{\infty}$ operad acts on the space $\Oinfty X$.  In particular, there is a map
$$ \mu: \C_{\infty}(2) \times_{\Z/2} (\Oinfty X)^2 \ra \Oinfty X.$$
The next theorem is our key geometric input.  It is quite easily deduced from \cite[Thm.1.10]{ak}, and hopefully seems plausible.  See Appendix \ref{tower appendix} for a bit more detail.

\begin{thm} \label{operad action on tower thm} There is a weak natural transformation of towers
$$\mu: (P\sm P)_{h\Z/2} \ra P$$ with the following properties. \\

\noindent{\bf (a)} There is a commutative diagram of weak natural transformations
\begin{equation*}
\xymatrix{
(\Sinfty_+ (\Oinfty X)^2)_{h\Z/2} \ar[d]^{\mu} \ar[rr]^-{(e\sm e)_{h\Z/2}} && (P\sm P)_{h\Z/2} \ar[d]^{\mu}  \\
\Sinfty_+ \Oinfty X \ar[rr]^{e} && P. }
\end{equation*}

\noindent{\bf (b)} On fibers, $\mu$ corresponds to the maps $D_b \sm D_c \ra D_{b+c}$ and $D_2D_d \ra D_{2d}$.
\end{thm}

\begin{proof}[Proof of \propref{thm1c prop}]  We have a commutative diagram of weak natural transformations
\begin{equation*}
\xymatrix{
\D_2(\Sinfty_+ \Oinfty X) \ar[r]^-{\D_2e} & \D_2(P_{d}) &  \D_2(D_{d}) \ar[l]_{\D_2i}   \\
\D_2(\Sinfty_+ \Oinfty X) \ar[r] \ar@{=}[u] \ar[d]& t_{\Z/2}((P\sm P)_{2d}) \ar[u]_{\wr} \ar[d]&  t_{\Z/2}(F_{2d}) \ar[l] \ar[u]_{\wr} \ar[d]  \\
D_2(\Sinfty_+ \Oinfty X) \ar[r]  \ar[d]& ((P\sm P)_{2d})_{h\Z/2}  \ar[d]&  (F_{2d})_{h\Z/2} \ar[l] \ar[d]  \\
\Sinfty_+ \Oinfty X \ar[r]^-{e} & P_{2d} &  D_{2d}. \ar[l]_{i}}
\end{equation*}
Here the bottom squares commute by the last theorem, and the right two top vertical maps are weak equivalences by \lemref{Tate vanishing stuff}.
\end{proof}

\begin{proof}[Proof of \propref{thm1b prop}]  This time we have a commutative diagram of weak natural transformations
\begin{equation*}
\xymatrix{
\D_{2}D_d \ar[r]^-{\D_2i} & \D_2 P_{d} &  \D_2 P_{d+r-1} \ar[r]^-{\D_2\delta} \ar[l]_-{\D_2p} & \Sigma \D_2D_{d+r}   \\
t_{\Z/2}(F_{2d}) \ar[r] \ar[d] \ar[u]_{\wr} & t_{\Z/2}((P \sm P)_{2d}) \ar[d] \ar[u]_{\wr}&  t_{\Z/2}((P \sm P)_{2d+2r-1}) \ar[r] \ar[l] \ar[d] \ar[u]_{\wr}& \Sigma t_{\Z/2}(F_{2d+2r}) \ar[d]  \ar[u]_{\wr}\\
(F_{2d})_{h\Z/2} \ar[r] \ar[d] & ((P \sm P)_{2d})_{h\Z/2} \ar[d] &  ((P \sm P)_{2d+2r-1})_{h\Z/2} \ar[r] \ar[l] \ar[d] & \Sigma (F_{2d+2r})_{h\Z/2} \ar[d]  \\
D_{2d} \ar[r]^-{i} & P_{2d} &  P_{2d+2r-1}\ar[r]^-{\delta} \ar[l]_-{p}& \Sigma D_{2d+2r}.}
\end{equation*}
Again the top  vertical maps are weak equivalences by \lemref{Tate vanishing stuff}.
\end{proof}

\section{Derived functors of destabilization} \label{derived functors section}

In this section, we will carefully define the Singer complex
$$ R_0M \xra{d_0} R_1M \xra{d_1} R_2M \xra{d_2} R_3M \xra{d_3} \dots$$
of the introduction and prove \thmref{Rs thm}, which says that the homology of this complex computes the derived functors $\Oinfty_s M$ for all $M \in \M$.

As a free standing theorem, \thmref{Rs thm} is similar (and maybe identical) to \cite[Thm.3.17]{goerss}.  Goerss works totally algebraically, and at key moments his proof appeals to computations and ad hoc arguments by others including Singer \cite{singer}, Brown and Gitler \cite{brown gitler}, and Bousfield et.~al.\ \cite{six author}.  By contrast, we give a geometrically based construction of this chain complex, with explicit calculations bypassed by appealing to our knowledge of the homology of the extended powers.

Our proof of \thmref{Rs thm} makes use of the doubling functor $\Phi: \M \ra \M$, dual to Powell's use of it in the cohomological setting \cite{powell}.  Also included in this section is a presentation of properties of $\Phi$ and $\Omega: \U \ra \U$ needed in our construction of the algebraic spectral sequence in \secref{alg spec seq section}.  Much of what we say about these functors is dual to cohomological presentations in \cite{lz} and \cite{s}.

\subsection{Injective resolutions in $\M$}

We say a bit about injectives in the category $\M$.

Since modules in $\M$ are locally finite, they are certainly also locally Noetherian.  The abelian category $\M$ is thus a locally Noetherian abelian category satisfying good exactness properties, and so one knows a priori \cite[IV.2]{gabriel} that arbitrary direct sums of injectives in $\M$ are again injective, and injectives can be written essentially uniquely as the direct sum of indecomposable injectives. It is useful for us to show this explicitly.

Let $\A_* \in \M$ be the dual of $\A$, so $\A_* = H_*(H\Z/2)$.  Let $\V$ denote the category of $\Z$--graded vector spaces.  Given $V \in \V$, we let $IV = V \otimes A_*$.  Note that $\epsilon: \A_* \ra \Z/2$ induces a map of graded vector spaces $\epsilon_V: IV \ra V$.

\begin{lem} For all $M \in \M$, the natural map
$$ \epsilon_{M,V}: \Hom_{\M}(M, IV) \simeq \Hom_{\V}(M,V)$$
sending $f$ to $\epsilon_V \circ f$ is an isomorphism.
\end{lem}
\begin{proof}[Sketch proof]  If $M$ is finite, $\epsilon_{M,\Sigma^n \Z/2}: \Hom_{\M}(M,\Sigma^n \A_*) \simeq (M_n)^{\#}$
is readily checked to be an isomorphism, and thus the same is true for $\epsilon_{M,V}$ when $M$ is finite and $V$ is finite dimensional.

For finite $M$ and arbitrary $V$, one then sees that $\epsilon_{M,V}$ is an isomorphism by filtering $V$ by its finite dimensional subspaces.

For arbitrary $M$ and $V$, one then sees that $\epsilon_{M,V}$ is an isomorphism by filtering $M$ by its finite submodules.
\end{proof}

\begin{cor}  The modules $IV$ are injective objects of $\M$, and every $M \in \M$ admits an injective resolution of the form
$$ 0 \ra M \ra IV(0) \ra IV(1) \ra IV(2) \ra \dots,$$
for some graded vector spaces $V(s) \in \V$.
\end{cor}
\begin{proof} As the functor sending $M$ to $\Hom_{\V}(M,V)$ is exact, we conclude that $IV$ is injective in $\M$.

Given $M \in \M$, the $\A$--module map $M \ra IM$ corresponding to $1_M \in \Hom_{\V}(M,M)$ is clearly monic.  It follows that injective resolutions of the asserted sort exist.
\end{proof}

It follows that every injective in $\M$ is a direct sum of modules of the form $\Sigma^n \A_*$, and thus is isomorphic to $IV$ for some $V \in \V$.

\subsection{Exact functors from the category $\M$ via topology}

Let $H_*(\Sp) \subset \M$ be the subcategory obtained as the image of $H_*: \Sp \ra \M$.  Thus the objects are the locally finite $\A$--modules of the form $H_*(X)$, with morphisms all $\A$--module maps of the form $f_*: H_*(X) \ra H_*(Y)$ for some $f: X \ra Y$.

Let $\Ab$ be an abelian category, for example $\M$.  Call a functor $F: H_*(\Sp) \ra \Ab$ {\em homological} if whenever $X \ra Y \ra Z$ induces a sequence $H_*(X) \ra H_*(Y) \ra H_*(Z)$ that is exact at $H_*(Y)$, then $F(H_*(X)) \ra F(H_*(Y)) \ra F(H_*(Z))$ is exact at $F(H_*(Y))$.  The following is a useful way to construct exact functors from $\M$, and natural transformations between such.

\begin{prop} \label{exact functor prop} {\bf (a)} Any homological functor $F: H_*(\Sp) \ra \Ab$ extends uniquely to an exact functor $F: \M \ra \Ab$. \\

\noindent{\bf (b)}  Let $F,G: H_*(\Sp) \ra \Ab$ be homological.  Any natural transformation $\phi: F \ra G$ extends uniquely to a natural transformation between the extended functors.
\end{prop}

To prove this, we first note that injectives in $\M$ can be topologically realized: given $V \in \V$, there is an associated generalized Eilenberg--MacLane spectrum $HV$, satisfying $\pi_*(HV) = V$ and $H_*(HV) = IV$.  The next lemma is clear.

\begin{lem} For all spectra $X$, the Hurewitz map induces an isomorphism
$$ [X, HV] \simeq \Hom_{\M}(H_*(X),IV).$$
\end{lem}

\begin{proof}[Proof of \propref{exact functor prop}]  Given a homological functor $F: H_*(\Sp) \ra \Ab$, we define its extension $F: \M \ra \Ab$ as follows.  Given $M \in \M$, choose an exact sequence
$0 \ra M \ra H_*(HV(0)) \xra{f_*} H_*(HV(1))$, and then define $F(M)$ to be the kernel of $F(f_*)$.  Given a morphism $\alpha: M \ra N$ in $\M$, one can construct a diagram
\begin{equation*}
\xymatrix{
0 \ar[r] & M \ar[d]^{\alpha} \ar[r] & H_*(HV(0)) \ar[d]^{\beta_*}  \ar[r]^{f_*} & H_*(HV(1)) \ar[d]^{\gamma_*}  \\
0 \ar[r] & N  \ar[r] & H_*(HW(0))   \ar[r]^{g_*} & H_*(HW(1))   }
\end{equation*}
with exact rows.  Applying $F$ to the right square and taking kernels, defines a map $F(\alpha): F(M) \ra F(N)$.
It is routine to check that this gives a well defined exact functor which extends the original functor up to natural isomorphism. This proves (a).  The proof of (b) is similar.
\end{proof}

\subsection{A topological definition of $\mathcal R_sM$} \label{top def section}

We construct various functors and natural transformations using the method of \propref{exact functor prop}.

\begin{defn}  Define $\mathcal R_s: H_*(\Sp) \ra \M$ by the formula
$$ \mathcal R_sH_*(X) = \im \{\epsilon_*: H_*(\Sigma D_{2^s}(\Sigma^{-1} X) \ra H_*(D_{2^s}X)\}.$$
\end{defn}

Thanks to our knowledge of $\epsilon_*$ as summarized in  \lemref{epsilon lemma}, we see that
$$\mathcal R_sH_*(X) = \langle Q^Ix \ | \ l(I) = s,  x \in H_*(X) \rangle/(\text{unstable and Adem relations}).$$

We remind readers that the Dyer--Lashof Adem relations are
$$Q^r Q^s = \sum_i \binom{i-s-1}{2i-r} Q^{r+s-i} Q^i,$$
and that the Steenrod algebra acts via the Nishida relations
$$(Q^sx)Sq^r = \sum_i \binom{s-r}{r-2i} Q^{s-r+i}(x Sq^i).$$

\begin{lem} $\mathcal R_s$ is homological.
\end{lem}
\begin{proof}  This follows immediately from the observation that the natural map of graded vector spaces
$$ \bigoplus_{n \in \Z} \mathcal R_s(\Sigma^n \Z/2) \otimes H_n(X) \ra \mathcal R_s(H_*(X))$$
sending $Q^I\iota_n \otimes x$ to $Q^Ix$ is an isomorphism. (In this formula, $H_n(X)$ should be regarded as  just a degree 0 vector space.)
\end{proof}

\begin{defns} {\bf (a)} Let $\mathcal R_s: \M \ra \M$ be the exact extension of $\mathcal R_s: H_*(\Sp) \ra \M$. \\

\noindent{\bf (b)}   Let $\epsilon: \Sigma \mathcal R_sM \ra \mathcal R_s\Sigma M$ be the natural $\A$--module map induced by the natural transformation $\epsilon: \Sigma D_{2^s}X \ra D_{2^s}\Sigma X$. \\

\noindent{\bf (c)}   Let $\mu: \mathcal R_s \mathcal R_t M \ra \mathcal R_{s+t} M$ be the natural $\A$--module map induced by the natural transformation $\mu: D_{2^s}D_{2^t}X \ra D_{2^{s+t}}X$. \\

\noindent{\bf (d)} Let $Q^i: (\mathcal R_sM)_n \ra (\mathcal R_{s+1}M)_{n+i}$ be the natural $\Z/2$--linear map induced by the Dyer--Lashof operation $Q^i: H_n(D_{2^s}X) \ra H_{n+i}(D_{2^{s+1}}X)$.
\end{defns}

We can immediately deduce lots of properties of these natural transformations. We note, in particular, a couple.

\begin{lem} \label{mu e lemma} {\bf (a)} The operations $Q^i$ satisfy the Adem relations, the Nishida relations, and the Dyer--Lashof unstable relation. \\

\noindent{\bf (b)} The diagram
\begin{equation*}
\xymatrix{
\Sigma \mathcal R_s \mathcal R_t M \ar[d]^{\mu} \ar[r]^{\epsilon} & \mathcal R_s \Sigma \mathcal R_t M \ar[r]^{\epsilon} & \mathcal R_s \mathcal R_t \Sigma M \ar[d]^{\mu}  \\
\Sigma \mathcal R_{s+t} M \ar[rr]^{\epsilon} && \mathcal R_{s+t} \Sigma M }
\end{equation*}
commutes.
\end{lem}

\begin{rem}  Observe that $\mathcal R_0(M) = M$, and $\epsilon: \Sigma \mathcal R_0M \ra \mathcal R_0\Sigma M$ is just the identity map on $\Sigma M$.
\end{rem}

Another elementary property we will need involves connectivity.

\begin{lem} \label{conn lemma} If $M$ is $(n-1)$ connected, then $\mathcal R_s M$ is $2^sn-1$ connected.
\end{lem}
\begin{proof} This follows from the observation that if $X$ is $(n-1)$ connected, then $D_dX$ is $dn-1$ connected.
\end{proof}

Now we introduce algebraic differentials. As before, $\delta: X \ra \Sigma D_2 X$ is the connecting map of the cofibration sequence $D_2 X \ra \tilde{P}_2 X \ra X$.

\begin{defn}  Define $d_s: \mathcal R_s(M) \ra \mathcal R_{s+1}(\Sigma M)$ to be the natural transformation induced by
the composite
$$ \delta_s: D_{2^s}X \xra{D_{2^s}\delta} D_{2^s}\Sigma D_2X \xra{D_{2^s}\epsilon} D_{2^s}D_2\Sigma X \xra{\mu}
 D_{2^{s+1}}\Sigma X.$$
\end{defn}

Explicitly, the computation of $\delta_*$ given in \thmref{theorem 1}(a) tells us that
$$ d_s(Q^Ix) = \sum_{i \geq 0} Q^IQ^{i-1}(\sigma xSq^i).$$

\begin{prop} \label{complex prop} The composite $$\mathcal R_{s-1}(\Sigma^{-1}M) \xra{d_{s-1}} \mathcal R_{s}(M) \xra{d_{s}}\mathcal R_{s+1}(\Sigma M)$$ is zero.
\end{prop}

This is an immediate consequence of the following topological version, and since homology is compactly supported, we really just need this result when $X$ is a finite CW spectrum.

\begin{prop}  \label{top complex prop} The composite $$D_{2^{s-1}}(X) \xra{\delta_{s-1}}D_{2^s}(\Sigma X) \xra{\delta_{s}}
 D_{2^{s+1}}(\Sigma^2 X)$$ is null.
\end{prop}
\begin{proof} It is easy to see that this composite factors through $D_{2^{s-1}}$ applied to the composite
$$ X \xra{\delta_0} D_2(\Sigma X) \xra{\delta_1} D_4(\Sigma^2 X).$$
Thus we just need to show that this last composite is null.

The trick now is to colinearize these functors and maps.  Generalizing our previous notation $\D_2$, for any $d$, let $\D_d(X) = \holim_n \Sigma^n D_d(\Sigma^{-n}X)$.

Colinearization then yields a commutative diagram of weak natural transformations
\begin{equation*}
\xymatrix{
\D_{1}(X) \ar[d] \ar[r] & X \ar[d]^{\delta_{0}}  \\
\D_{2}(\Sigma X) \ar[d] \ar[r] & D_{2}(\Sigma X) \ar[d]^{\delta_1}  \\
\D_4(\Sigma^2X)  \ar[r] & D_{4}(\Sigma^2 X).  }
\end{equation*}

As the top horizontal map is clearly an equivalence, the proposition will follow if we can show the left composite is null.

We offer two rather different reasons for this.

The first argument only seems to hold when $X$ is finite, and depends on consequences of the Segal Conjecture for elementary abelian 2--groups. Namely, the first author showed \cite[Cor.5.3]{kuhn chevalley} that $\D_4(X) \simeq *$ if $X$ is finite. (In this case, it is also true that the top left vertical map is an equivalence after completing at 2.)

A second, more elementary argument goes roughly as follows.  The colinearized functors $\D_d$ preserves cofibration sequences, and are null unless $d$ is a power of 2.  Then it is not too hard to show that the left vertical sequence is equivalent to the composite
$$ \D_1(X) \ra \Sigma \D_2(X) \ra \Sigma^2 \D_4(X)$$
of the two connecting maps associated to the colinearization of the tower
$$ \tilde P_4(X) \ra \tilde P_2(X) \ra  \tilde P_1(X),$$
so that their composite is null.
\end{proof}

\begin{rem}  A direct algebraic proof of \propref{complex prop} is possible.  Using both the Dyer--Lashof Adem relations and the Adem relations in $\A$, one needs to show that
$$ \sum_{i\geq 0} \sum_{j \geq 0} Q^{i-1}Q^{j-1}(x Sq^i Sq^j) = 0.$$
Goerss \cite[Lem.3.13]{goerss} points to Brown and Gitler's assertion that a calculation like this is straightforward \cite[Lem.2.3]{brown gitler}, and one can check that it is.
\end{rem}

\subsection{The doubling functor and $\mathcal R_*(M)$}

In the cohomological setting, the following definition should be familiar to readers of \cite{lz} and \cite{s}.
\begin{defn}  If $M \in \M$, $\Phi(M) \in \M$ is defined to be the module concentrated in even degrees, with $\Phi(M)_{2n} = M_n$ and with $\phi(x)Sq^{2i} = \phi(xSq^i)$.  (Here, given $x \in M_n$, we have written $\phi(x)$ for the corresponding element in $\Phi(M)_{2n}$.)
\end{defn}

Basic properties are listed in the next lemma.

\begin{lem} {\bf (a)} $\Phi$ is an exact functor preserving unstable modules. \\

\noindent{\bf (b)} $\Phi(N \otimes M) = \Phi(N) \otimes \Phi(M)$.  In particular, $\Phi(\Sigma M) = \Sigma^2\Phi(M)$. \\

\noindent{\bf (c)} Let $\Gamma^2(M) = (M \otimes M)^{\Z/2}$ and $S^2(M) = (M\otimes M)_{\Z/2}$. The composite  $\Gamma^2(M) \hookrightarrow M\otimes M \twoheadrightarrow S^2(M)$ naturally factors as a composite  $$\Gamma^2(M) \twoheadrightarrow \Phi(M) \hookrightarrow S^2(M),$$ where the second map sends $\phi(x)$ to $x^2$. \\

\noindent{\bf (d)} Let $sq_0: M \ra \Phi(M)$ be the linear map defined by letting $sq_0(x) = \phi(xSq^n)$ if $x \in M_{2n}$.  If $M$ is unstable, then $sq_0$ is $\A$--linear.
\end{lem}

For a proof of (d) in the cohomological setting, see \cite[p.26]{s}.

\begin{defn}  Let $q_0: \Phi(M) \ra \mathcal R_1M$ be defined by the formula $q_0(\phi(x)) = Q^{|x|}x$.  More generally, define $q_0: \Phi(\mathcal R_{s}M) \ra \mathcal R_{s+1}M$ to be the composite
$ \Phi(\mathcal R_{s}M) \xra{q_0} \mathcal R_1 \mathcal R_s M \xra{\mu} \mathcal R_{s+1}M$.
\end{defn}

\begin{lem}  $q_0$ is $\A$--linear.
\end{lem}
\begin{proof} As usual, one need just check this when $M = H_*(X)$.  The identity $Q^{|x|}x = x^2 \in H_*(D_2X) $ for all $x \in H_*(X)$ implies that the composite $$\Phi(H_*(X)) \xra{q_0} \mathcal R_1(H_*(X)) \subseteq H_*(D_{2}X)$$ equals the composite $$\Phi(H_*(X)) \hra S^2(H_*(X))  \ra H_*(D_{2}X),$$ and so is $\A$--linear.
\end{proof}

The following lemma is crucial.

\begin{lem} \label{ses lem} For all $M \in \M$ and $s>0$, the sequence
$$ 0 \ra \Phi(\mathcal R_{s-1}(M)) \xra{q_0} \mathcal R_s(M) \xra{\epsilon} \Sigma^{-1}\mathcal R_s(\Sigma M) \ra 0$$
is short exact. \\
\end{lem}

\begin{proof}  It is convenient to use lower indices for Dyer--Lashof operations: $Q_ix = Q^{|x|+i}x$.  Suppose $M$ has a homogeneous basis $\{x_{\alpha}\}$.  Then the Adem relations show that $\mathcal R_s(M)$ then has a basis given by
$$ \{ Q_{i_0}Q_{i_1} \dots Q_{i_s}x_{\alpha} \ | \ 0 \leq i_0 \leq i_1 \leq \dots \leq i_s\}.$$
Since
$$ \epsilon( Q_{i_0}Q_{i_1} \dots Q_{i_s}x_{\alpha}) = \sigma^{-1} Q_{i_0-1}Q_{i_1-1} \dots Q_{i_s-1}\sigma x_{\alpha},$$
and $Q_ix = 0$ if $i<0$, the lemma follows.
\end{proof}

The next lemma is clear from the definitions.

\begin{lem} \label{chain maps lemma}
\noindent{\bf (a)} The diagram
\begin{equation*}
\xymatrix{
\mathcal R_{s}\mathcal R_t (M) \ar[d]^{\mathcal R_s (d_{t})} \ar[r]^{\mu} & \mathcal R_{s+t}(M) \ar[d]^{d_{s+t}}  \\
\mathcal R_{s}\mathcal R_{t+1} (\Sigma M) \ar[r]^{\mu} & \mathcal R_{s+t+1}(\Sigma M) }
\end{equation*}
commutes. \\

\noindent{\bf (b)} The diagram
\begin{equation*}
\xymatrix{
\Phi(\mathcal R_{s-1}(M)) \ar[d]^{\Phi(d_{s-1})} \ar[r]^{q_0} & \mathcal R_s(M) \ar[d]^{d_s}  \ar[r]^-{\epsilon} & \Sigma^{-1}\mathcal R_s(\Sigma M) \ar[d]^{ \Sigma^{-1} d_s}\\
\Phi(\mathcal R_{s}(\Sigma M)) \ar[r]^{q_0} & \mathcal R_{s+1}(\Sigma M) \ar[r]^-{\epsilon} & \Sigma^{-1}\mathcal R_{s+1}(\Sigma^2 M) }
\end{equation*}
commutes.
\end{lem}

\subsection{The derived functors of destabilization}

We now relabel as in the introduction.

\begin{defn} Let $R_s = \Sigma \mathcal R_s\Sigma^{s-1}: \M \ra \M$.
\end{defn}

With this notation, the chain complex
$$ \Sigma \mathcal R_0(\Sigma^{-1}M) \xra{d_0} \Sigma \mathcal R_1(M) \xra{d_1} \Sigma \mathcal R_2(\Sigma M) \xra{d_2} \Sigma \mathcal R_3(\Sigma^2 M) \ra \dots$$
rewrites as
$$ R_0(M) \xra{d_0} R_1(M) \xra{d_1}R_2(M) \xra{d_2}R_3(M)  \ra \dots.$$

The following is a restatement of \thmref{Rs thm}.

\begin{thm} \label{Rs thm: 2} For all $M \in \M$, there is a natural isomorphism $$H_s(R_*(M); d_*) \simeq \Oinfty_s M.$$
\end{thm}

In the usual way, this theorem is a consequence of the next three lemmas.

\begin{lem} $R_s$ is exact for all $s$.
\end{lem}

\begin{lem} \label{R0 lemma} $H_0(R_*(M); d_*) = \Oinfty M$.
\end{lem}

\begin{lem} \label{acyclicity lemma} For all $n \in \Z$ and $s>0$, $H_s(R_*( \Sigma^n \A_*); d_*)=0$.
\end{lem}

The first of these lemmas is evident, and we quickly check the second.
\begin{proof}[Proof of \lemref{R0 lemma}]
We need to compute the kernel of $d_0: M \ra \Sigma \mathcal R_1M$, and we recall that
$$d_0(x) = \sum_{i\geq 0} \sigma Q^{i-1}(xSq^i).$$
Then
\begin{equation*}
\begin{split}
x \in \ker(d_0) &
\Leftrightarrow Q^{i-1}(xSq^i)=0 \text{ for all } i\geq 0 \\
  & \Leftrightarrow xSq^i=0 \text{ whenever } i-1\geq |x|-i \\
  & \Leftrightarrow xSq^i=0 \text{ whenever } 2i > |x|. \\
  &  \Leftrightarrow x \in \Oinfty M.
\end{split}
\end{equation*}
\end{proof}

The proof of \lemref{acyclicity lemma} will take a bit of preparation.  Firstly, \lemref{ses lem} and \lemref{chain maps lemma}(b) combine to tell us the following.

\begin{prop} \label{ses prop}  $ 0 \ra \Phi(R_{*-1}(\Sigma M)) \xra{q_0} \Sigma R_*(M)  \xra{\epsilon} R_*(\Sigma M) \ra 0$ is a short exact sequence of chain complexes.
\end{prop}

Temporarily, let $H_s(M) = H_s(R_*(M); d_*)$.  The short exact sequence of \propref{ses prop} induces a long exact sequence
\begin{multline*}
0 \ra \Sigma H_0(M) \xra{\epsilon_*} H_0(\Sigma M) \xra{\partial} \Phi(H_0(\Sigma M)) \xra{q_0} \Sigma H_1(M) \ra \dots \\
\dots \ra H_{s-1}(\Sigma M) \xra{\partial} \Phi(H_{s-1}(\Sigma M)) \xra{q_0} \Sigma H_s(M) \xra{\epsilon_*} H_s(\Sigma M) \ra \dots
\end{multline*}

We need to identify the first boundary map.

\begin{lem} \label{boundary = sq0 lem} $H_0(\Sigma M) \xra{\partial} \Phi(H_0(\Sigma M))$ identifies with the map  $$\Oinfty(\Sigma M) \xra{sq_0} \Phi(\Oinfty(\Sigma M)).$$
\end{lem}
\begin{proof}  If $\sigma x \in \Oinfty(\Sigma M)$ has $|\sigma x|=2n$, then we have the correspondence, under the maps $\Sigma M \xra{\Sigma d_0} \Sigma R_1(M) \xla{q_0} \Phi(\Sigma M)$,
$$(\Sigma d_0)(\sigma x) = \sigma d_0(x) = \sigma Q^{n-1}(xSq^n) = q_0(\phi(\sigma xSq^n)) = q_0(sq_0(\sigma x)).$$
Thus $\partial(\sigma x) = sq_0(\sigma x)$.
\end{proof}

In dual form, the following lemma corresponds to the familiar fact that the map $Sq_0: F(n) \ra F(n)$, sending $x$ to $Sq^{|x|}x$, is monic.  Here $F(n)$ is the free unstable $\A$--module on an $n$--dimensional class.

\begin{lem} \label{reduced lemma} For all $n \in \Z$, $\Oinfty(\Sigma^n A_*) \xra{sq_0} \Phi(\Oinfty(\Sigma^n A_*))$ is onto.
\end{lem}

We are finally ready to prove \lemref{acyclicity lemma}.  The proof is dual to the proof of \cite[Prop.9.4.1]{powell}.

\begin{proof}[Proof of \lemref{acyclicity lemma}]  By induction on $s\geq 1$, we prove that $H_s(\Sigma^n A_*) = 0$.  In all cases, we consider the exact sequence
$$ H_{s-1}(\Sigma^{n+1} A_*) \xra{\partial} \Phi(H_{s-1}(\Sigma^{n+1} A_*)) \xra{q_0} \Sigma H_s(\Sigma^n A_*) \xra{\epsilon_*} H_s(\Sigma^{n+1} A_*).$$
In the initial case when $s=1$, the previous two lemmas show that $\partial$ is onto.  If $s>1$, then, under the inductive hypothesis, $\Phi(H_{s-1}(\Sigma^{n+1} A_*)) = 0$.  Thus, in all cases, we can conclude that $\Sigma H_s(\Sigma^n A_*) \xra{\epsilon_*} H_s(\Sigma^{n+1} A_*)$ is monic for all $n$.  But, by \lemref{conn lemma}, the connectivity of $\Sigma^{-m}H_s(\Sigma^{m+n}\A_*)$ is at least $(2^s-1)m+2^s(n+s-1)$, and so goes to infinity as $m$ goes to infinity.
\end{proof}

\subsection{First consequences}

\thmref{Rs thm: 2}, when combined with \propref{ses prop} and \lemref{boundary = sq0 lem}, implies the following.

\begin{cor} \label{les cor 2}  For all $M \in \M$, there is a natural long exact sequence
\begin{multline*}
0 \ra \Sigma \Oinfty_0(M) \xra{\epsilon_*} \Oinfty_0(\Sigma M) \xra{sq_0} \Phi(\Oinfty_0(\Sigma M)) \xra{q_0} \Sigma \Oinfty_1(M) \ra \dots \\
\dots \ra \Oinfty_{s-1}(\Sigma M) \xra{sq_0} \Phi(\Oinfty_{s-1}(\Sigma M)) \xra{q_0} \Sigma \Oinfty_s(M) \xra{\epsilon_*} \Oinfty_s(\Sigma M) \ra \dots
\end{multline*}
\end{cor}

\begin{rem}  This long exact sequence already appears (in dual form) in \cite[\S 4.1]{lz}. One observes that, if $M \ra I_*(M)$ is an injective resolution in $\M$, then so is $\Sigma M \ra \Sigma I_*(M)$, and
$$ 0 \ra \Sigma \Oinfty(I_*(M)) \xra{\epsilon} \Oinfty(\Sigma I_*(M)) \xra{sq_0} \Phi(\Oinfty(\Sigma I_*(M))) \ra 0$$
is short exact.  This short exact sequence of chain complexes then induces the long exact sequence of the corollary.

It is amusing that $sq_0$ identifies with the boundary map in our derivation, while $q_0$ identifies with the boundary map in the Lannes--Zarati approach.
\end{rem}

Next we note that \lemref{conn lemma} implies the following general connectivity estimate.

\begin{cor}  If $M \in \M$ is $n$--connected, then $\Oinfty_s(M)$ is at least $2^s(n+s)$--connected. For all $M \in \M$, $\displaystyle \colim_n \Sigma^{-n} \Oinfty_s(\Sigma^n M) = 0$ for all $s \geq 1$.
\end{cor}

The reasoning we gave in the proof of \lemref{acyclicity lemma} then proves the following useful criterion for the vanishing of the higher derived functors.

\begin{prop} \label{vanishing prop}  If $sq_0:\Oinfty(\Sigma^n M) \ra \Phi(\Oinfty(\Sigma^n M))$ is onto for all $n \geq 1$, then $\Oinfty_s(M) = 0$ for all $s \geq 1$.
\end{prop}

Following \cite{lz} and \cite{goerss}, we now deduce some properties of $\Oinfty_s(\Sigma^{-t}M)$ when $M$ is unstable.

\begin{lem} Suppose $M$ is unstable. Then, for all $s \geq 0$, $\mathcal R_s(M)$ is also unstable, and $d_{s}: \mathcal R_{s}(\Sigma^{-1}M) \ra \mathcal R_{s+1}(M)$ is zero.
\end{lem}

\begin{proof}  Though this admits an algebraic proof, to show how the algebra follows the topology, we offer a topologically based proof.

If $M \subset H_*(HV)$ and is unstable, then $M \hookrightarrow \Oinfty H_*(HV) \twoheadleftarrow H_*(\Oinfty HV)$.  Using the exactness of the functors $\mathcal R_s$, one easily sees that the conclusions of the lemma for the module $H_*(\Oinfty HV)$ imply the same for $M$.  Thus it suffices to prove the lemma when $M = H_*(Z)$, where $Z$ is a space.

In this case, $\mathcal R_s(H_*(Z)) \subset H_*(D_{2^s}Z)$, which is unstable, as $D_{2^s}Z$ is a space.

To see that $d_{s}: \mathcal R_{s}(\Sigma^{-1}H_*(Z)) \ra \mathcal R_{s+1}(H_*(Z))$ is zero, we recall that it is induced by a geometric stable map $\delta_s: D_{2^s}(\Sigma^{-1}Z) \ra D_{2^{s+1}}(Z)$.  (We identify $Z$ with $\Sinfty Z$.)  We observe that this map is null: $\delta_s$ factors through $D_{2^s}(\delta_0)$, and $\delta_0: \Sigma^{-1}Z \ra D_2(Z)$ is null as $\Sigma \delta_0$ is the first boundary map in the tower associated to $\Sinfty \Oinfty \Sinfty Z$, which splits into the product of its fibers.
\end{proof}

As $\Oinfty_s(\Sigma^{1-s}M)$ is the homology at the middle term of the complex
$$ \Sigma \mathcal R_{s-1}(\Sigma^{-1}M) \xra{d_{s-1}} \Sigma \mathcal R_s(M) \xra{d_s} \Sigma \mathcal R_{s+1}(\Sigma M),$$
the lemma leads to the next result.

\begin{thm} \label{unstable thm} Suppose $M$ is unstable. \\

\noindent{\bf (a)}  $\Oinfty_s(\Sigma^{1-s}M) \simeq \Sigma \mathcal R_s(M)$, so that
$\Omega \Oinfty_s(\Sigma^{1-s}M) \simeq \mathcal R_s(M)$. \\

\noindent{\bf (b)} More generally, if $s>t$, then $\Oinfty_s(\Sigma^{-t}M) \simeq \Sigma \mathcal R_s(\Sigma^{s-t-1}M)$, which is a quotient of $\Sigma^{s-t} \mathcal R_s(M)$.  Thus $\Oinfty_s(\Sigma^{-t}M)$ is an $(s-t)$--fold suspension of an unstable module, and so
$ \Omega^{s-t}\Oinfty_s(\Sigma^{-t}M) = \Sigma^{-s + t+ 1}\mathcal R_s(\Sigma^{s-t-1}M)$. \\

\noindent{\bf (c)} $\Oinfty_s(\Sigma^{-s}M) \simeq \coker \{ d_{s-1}: \Sigma \mathcal R_{s-1}(\Sigma^{-2}M) \ra \Sigma \mathcal R_{s}(\Sigma^{-1}M) \}$. \\
\end{thm}

The first statement here is the main algebraic theorem of \cite{lz}, and the last was observed in \cite[Cor.5.4]{goerss}.

\subsection{Dyer--Lashof operations on derived functors}

We need to explain \propref{DL structure on derived functors prop}, which said that the sum of the looped derived functors $\Omega \Oinfty_* \Sigma^{1-*}M$ is an object in $\QM$.  Otherwise said, we need to explain why there exist natural transformations
$$ \mu: \mathcal R_s \Omega \Oinfty_t \Sigma^{1-t}M \ra \Omega \Oinfty_{s+t} \Sigma^{1-s-t}M$$
compatible in the usual way.

Firstly we note that \lemref{chain maps lemma}(a) and \thmref{Rs thm: 2} together imply, when one is careful with suspensions, that the maps
$$\mu: \mathcal R_s \mathcal R_t M \ra \mathcal R_{s+t} M$$
induce maps
$$ \mu: \mathcal R_s \Sigma^{-1} \Oinfty_t \Sigma^{1-t}M \ra \Sigma^{-1} \Oinfty_{s+t} \Sigma^{1-s-t}M$$

We now need a better understanding of $\Omega \Oinfty_s(M)$ for general $M \in \M$. The following lemma is dual to \cite[Prop.1.7.5]{s}.

\begin{lem} \label{omega 1 lemma} $\Omega: \U \ra \U$ has only one nonzero right derived functor $\Omega_1$.  For all $M \in \U$, there is an exact sequence
$$ 0 \ra \Sigma \Omega M \ra M \xra{sq_0} \Phi(M) \ra \Sigma \Omega_1 M \ra 0.$$
\end{lem}

From the long exact sequence of \corref{les cor 2}, we thus deduce the following.
\begin{cor}
\label{image-epsilon-loop}
For $M \in \M$, the following diagram commutes, and the bottom row is short exact:
$$\xymatrix{
  \Sigma^{-1} \Phi \Omega^\infty_s (\Sigma M) \ar[dr]^{q_0} \ar@{->>}[d] & & \Sigma^{-1} \Omega^\infty_{s+1} (\Sigma M) \\
  \Omega_1 \Omega^\infty_s (\Sigma M) \ar@{ (->}[r] & \Omega^\infty_{s+1} (M) \ar@{->>}[r] \ar[ur]^{\epsilon_*} & \Omega \Omega^\infty_{s+1} (\Sigma M). \ar@{ (->}[u]
}$$
\end{cor}

\begin{proof}[Proof of \propref{DL structure on derived functors prop}]
From the commutative diagram of \lemref{mu e lemma},
\begin{equation*}
\xymatrix{
\Sigma \mathcal R_s \mathcal R_t M \ar[d]^{\mu} \ar[r]^-{\epsilon} & \mathcal R_s \Sigma \mathcal R_t M \ar[r]^-{\mathcal R_s \epsilon} & \mathcal R_s \mathcal R_t \Sigma M \ar[d]^{\mu}  \\
\Sigma \mathcal R_{s+t} M \ar[rr]^-{\epsilon} && \mathcal R_{s+t} \Sigma M, }
\end{equation*}
we deduce that the following diagram commutes:
\begin{equation*}
\xymatrix{
\Sigma \mathcal R_s \Sigma^{-1} \Oinfty_t \Sigma^{-t}M \ar[d]^{\mu} \ar@{->>}[r]^-{\epsilon} & \mathcal R_s \Oinfty_t \Sigma^{-t}M \ar[r]^-{\mathcal R_s \epsilon_*} & \mathcal R_s \Sigma^{-1} \Oinfty_t \Sigma^{1-t}M \ar[d]^{\mu}  \\
\Oinfty_{s+t} \Sigma^{-s-t}M \ar[rr]^-{\epsilon_*} && \Sigma^{-1} \Oinfty_{s+t} \Sigma^{1-s-t}M. }
\end{equation*}

For $M \in \M$, we then define
$$ \mu: \mathcal R_s \Omega \Oinfty_t \Sigma^{1-t}M \ra \Omega \Oinfty_{s+t} \Sigma^{1-s-t}M$$
to be the natural transformation induced by taking the image of the top and bottom horizontal maps in this last diagram.

These natural transformations for all $s$ and $t$ are equivalent to defining natural Dyer--Lashof operations
$$Q^i: \Omega \Oinfty_s\Sigma^{1-s}M \ra \Omega \Oinfty_{s+1}\Sigma^{-s}M,$$
for all $i \in \Z$, which raise degree by $i$, and satisfy the usual properties.
\end{proof}

\begin{defn} \label{q_0 defn} Define $q_0: \Phi(\Omega \Oinfty_s(\Sigma M)) \ra \Omega \Oinfty_{s+1}(M)$ by the formula $q_0(\phi(x)) = Q^{|x|}x.$
\end{defn}

\section{Hopf algebras and the algebraic spectral sequence} \label{alg spec seq section}

In this section, we first discuss part (e) of \thmref{global ss thm}, which said that each $E^r_{*,*}(X)$ is primitively generated, with its bigraded module of primitives an $\A$--module  subquotient of $\mathcal R_*H_*(X)$.  This is really a reflection of an aspect of the general theory of differential Hopf algebras.

Using some of the same theory, we then go on to develop the algebraic spectral sequence, as described in \thmref{theorem 2}.

The Hopf algebras of this paper, $E^r_{*,*}(X)$, can be viewed as {\em connected} bicommutative Hopf algebras, by viewing $E^r_{-i,*}(X)$ as having grading $i$.  As these are our concern, in this section by the term Hopf algebra, we will mean a connected bicommutative Hopf algebra over $\Z/2$.

\subsection{Primitively generated Hopf algebras and barcode modules.}

We recall some classic observations about primitively generated Hopf algebras.

One has as examples $\Z/2[x_d]$ and $\Z/2[x_d]/(x^{2^s}_d)$, with $x_d$ primitive and homogeneous of some positive grading $d$.  The work of Milnor and Moore \cite{milnor moore} then tells us that any primitively generated Hopf algebra will be a tensor product of Hopf algebras of these types.

A more basis--free way to discuss primitively generated Hopf algebras is via modules of primitives.

If $A$ is a Hopf algebra, its module of primitives has the structure of a positively graded vector space $V$ equipped with a linear map which doubles grading $q: V_* \ra V_{2*}$.  We will call such a $V$ a {\em barcode} module, as it is the sort of $\Z/2[q]$--module appearing in the persistent homology literature \cite{zc}, where the classification is given by a barcode.

There is then an equivalence of categories
$$ \text{barcode modules} \simeq \text{primitively generated Hopf algebras}.$$
In one direction, the correspondence takes a barcode module $V$ to
$$U_q(V) = S^*(V)/(x^2 - q(x): x \in V).$$
In the other direction, the primitives of a Hopf algebra, together with the squaring map, form a barcode module.

\begin{exs}  $\Z/2[x_d] = U_q(W)$, where
$$W = \langle x_d, x_{2d},x_{4d},\dots \rangle, \text{ and } q(x_{2^td}) = x_{2^{t+1}d}.$$  Similarly $\Z/2[x_d]/(x^{2^s}_d) = U_q(V)$, where
$$V = \langle x_d,x_{2d}, \dots, x_{2^{s-1}d} \rangle, \text{ and } q(x_{2^td}) = \begin{cases}
x_{2^{t+1}d} & \text{if } t<s-1 \\ 0 & \text{if } t = s-1.
\end{cases}$$
\end{exs}

\begin{rem} \label{grading remark} In our spectral sequences, $V$ will be a graded object in $\M$, equipped with an $\A$--module map $q: \Phi(V_*) \ra V_{2*}$.  In this situation,  $U_q(V)$ becomes an object in $\HM$, bigraded by giving $V_{i,j}$ bigrading $(-i,i+j)$.
\end{rem}

\begin{ex} $E^1_{*,*}(X) = U_q(V)$, where
$$V_i = \begin{cases}
\mathcal R_sH_*(X) & \text{if } i = 2^s \\ 0 & \text{otherwise, }
\end{cases}$$
with $q(x) = Q^{|x|}x$.
\end{ex}

\subsection{The homology of certain differential Hopf algebras}

A {\em differential} Hopf algebra is a Hopf algebra equipped with a homogeneous differential of some degree $r$ (not necessarily $\pm 1$) which is both a derivation and a coderivation.  It is easy to see that such differentials on $U_q(V)$ correspond to homogeneous linear maps $d: V_* \ra V_{*+r}$ such that $d^2 = 0 = dq$.  We will call such pairs $(V,d)$ differential barcode modules.

Results in \cite{milnor moore} show that a Hopf algebra (bicommutative over $\Z/2$) is primitively generated if and only if squares are zero in the dual.  Browder \cite{browder} notes that, when such a Hopf algebra has a differential, this property will be preserved by taking homology.

It follows that given a differential barcode module $(V,d)$, the homology Hopf algebra $H(U_q(V);d)$ will have the form $U_q(W)$ for some new barcode module $W$.  The natural problem now is to try to determine $W$ from $(V,d)$.

It seems difficult to say something useful about the problem in this generality.  In particular, the following simple example shows that $W$ need not be a subquotient of $V$.

\begin{ex} Let $V = \langle x,y \rangle$, with $q=0$, and $dx=y$.  Then $U_q(V)$ is the exterior algebra $\Lambda(x,y)$ and $H(U_q(V);d) = \Lambda(xy) = U_q(W)$, where $W = \langle xy \rangle$.
\end{ex}

More generally, the identity $d(xdx) = qdx$ holds in $U_q(V)$, for any pair $(V,d)$, so that elements in $\ker q \cap \im d$ lead to cycles in $H_*(U_q(V);d)$ similar to the cycle $xy$ in this last example.

Our discovery is that the differential Hopf algebras $U_q(V)$ arising in our spectral sequences satisfy an extra condition avoiding this problem, and ensuring that the module of primitives in $H(U_q(V);d)$ {\em is} a natural subquotient of $V$.

\begin{prop} \label{Hopf alg prop}  Suppose a differential barcode module $(V,d)$ is nonzero only in degrees of the form $2^k$, and $d$ has nonzero component $d:V_{2^s} \ra V_{2^t}$ for some $s<t$.
If also
\begin{equation*}\tag{$\diamondsuit_t$} q: V_{2^k} \ra V_{2^{k+1}} \text{ is monic for } k \geq t,
\end{equation*}
then $H(U_q(V);d) \simeq U_q(\bar V)$, where $\bar V$ is the following explicit barcode module subquotient of $V$:
\begin{equation*}
\bar V_{2^k} =
\begin{cases}
V_{2^k} & \text{if } k <t \text{ and } k\neq s \\ \ker(d) & \text{if } k=s \\
V_{2^k}/q^{k-t}\im(d) & \text{if } k \geq t.
\end{cases}
\end{equation*}
In particular, $\bar V$ still satisfies condition $(\diamondsuit_t)$.
\end{prop}

We will prove this in the next subsection.

\begin{cor} \label{Hopf alg ss cor}  Suppose that one has a left half plane homological spectral sequence $\{E^r\}$ of differential objects in $\HM$ with $E^1 = U_q(V(1,0))$, where $V(1,0)$ is as in \remref{grading remark}.  If $V(1,0)$ is nonzero only in degrees of the form $2^k$ and satisfies condition $(\diamondsuit_1)$, then \\

\noindent{\bf (a)} \ $d^r$ can only be nonzero when $r = 2^t - 2^s$ with $s<t$, and \\

\noindent{\bf (b)} \  $E^{2^t-2^s} = U_q(V(t,s))$ where $V(t,s)$ is a subquotient of $V(1,0)$ satisfying condition $(\diamondsuit_t)$.\\
\end{cor}

\begin{proof}  Using the proposition, the corollary is proved by induction on the set of pairs $\{(t,s) \ | \ 0\leq s<t\}$ totally ordered by the value $2^t-2^s$. (Thus $(1,0)<(2,1)<(2,0)<(3,2)<\dots$.) Note that, in passing from the pair $(t,0)$ to $(t+1,t)$, one uses that $(\diamondsuit_t) \Rightarrow (\diamondsuit_{t+1})$.
\end{proof}

\begin{proof}[Proof of \thmref{global ss thm}(e) and (f)]  The corollary applies to $\{E^r_{*,*}(X)\}$, with $V(1,0)_{2^s} = \mathcal R_sH_*(X)$, as this $V(1,0)$ satisfies condition ($\diamondsuit_0$).
\end{proof}

\subsection{A useful lemma}

Given two positively graded vector spaces $U$ and $W$ and a degree doubling map $q: U \ra U \oplus W$, we let
$$\Gamma(q) = S^*(U \oplus W)/(q(x) - x^2),$$
a primitively generated Hopf algebra.

The map $q$ has components $q_U: U \ra U$ and $q_W: U \ra W$.  If a homogeneous map $d: U \ra W$ satisfies $dq_U = 0$, $d$ will induce a differential on $\Gamma(q)$.  It is useful to then let $U^{\prime} = \ker(d)$, $W^{\prime} = \coker(d)$, and $q^{\prime}: U^{\prime} \ra U^{\prime} \oplus W^{\prime}$ be the map induced by $q$.

There is an evident inclusion into the cycles
$$U^{\prime} \oplus W \hra Z_*(\Gamma(q);d),$$
and this induces a natural map of Hopf algebras
$$ \alpha(q,d):  \Gamma(q^{\prime}) \ra H_*(\Gamma(q);d).$$

\begin{lem} \label{useful lemma}  In this situation, $\alpha(q,d)$ is an isomorphism.
\end{lem}

\begin{proof}  We first consider the case when $q$ is identically zero.
As a Hopf algebra,
$$ \Gamma(0) = \Lambda^*(U) \otimes S^*(W).$$
We can also assume that our map $d: U \ra W$, say of degree r, has the form
$$ U^{\prime}\oplus B  \twoheadrightarrow B \xra{\sim} \Sigma^r B \hookrightarrow \Sigma^r B \oplus W^{\prime}.$$
Thus, as differential Hopf algebras,
$$ (\Gamma(0),d)  = \Lambda^*(U^{\prime}) \otimes (\Gamma_B,d_B) \otimes S^*(W^{\prime}),$$
where $(\Gamma_B,d_B) = \Lambda^*(B) \otimes S^*(\Sigma^r B)$ with the Koszul differential: $d_B(b\otimes 1) = 1 \otimes \sigma^r b$.  The complex $(\Gamma_B,d_B)$ is well known to be acyclic \cite[Cor.4.5.5]{weibel}, and easily checked to be:
it is the tensor product of complexes of the form $\Lambda^*(x) \otimes \Z/2[dx]$ whose homology is $\Z/2$.  Thus we see that
$$ H_*(\Gamma(0);d) = \Lambda^*(U^{\prime}) \otimes S^*(W^{\prime}) = \Gamma(0^{\prime}),$$
and thus the lemma is true for the case $q=0$.

Now we consider the case when just the component $q_W$ is assumed to be zero, so that $q=q_U$ for an arbitrary map $q_U: U \ra U$ satisfying $dq_U=0$.  We reduce to the previous case with a little spectral sequence argument.

We define an increasing filtration $F_0 \subset F_1 \subset F_2 \subset \dots $ of the chain complex $(\Gamma(q_U),d)$ by letting
$$ F_k = \im\{ S^{* \leq k}(U) \otimes S^*(W) \ra \Gamma(q)\}.$$

Then $d: F_k \ra F_{k-1}$, and in the induced spectral sequence $E^0 = E^1 = (\Gamma(0),d)$, so that $E^2 = \Gamma(0^{\prime})$.  Noting that $E^0(\Gamma(q_U^{\prime})) = \Gamma(0^{\prime})$, one simultaneously sees that $E^2 = E^{\infty}$ and $\alpha(q_U,d)$ is an isomorphism.

Finally we consider the case of a general $q = (q_U,q_W)$.  Once again we reduce to the previous case with a spectral sequence argument.

This time, we define a decreasing filtration $F^0 \supset F^1 \supset F^2 \supset \dots $ of the chain complex $(\Gamma(q),d)$ by letting
$$ F^k = \im\{ S^{*}(U) \otimes S^{*\geq k}(W) \ra \Gamma(q)\}.$$

Then $d: F^k \ra F^{k+1}$, and in the induced spectral sequence $E_0 = E_1 = (\Gamma(q_U),d)$, so that $E_2 = \Gamma(q_U^{\prime})$.  Noting that $E_0(\Gamma(q^{\prime})) = \Gamma(q_U^{\prime})$, one deduces that $E_2 = E_{\infty}$ and $\alpha(q,d)$ is an isomorphism.
\end{proof}

\begin{ex}  Here is an example that illustrates the various reductions made in the last proof.  Let $U = \langle x_1,x_2,x_4\rangle$, $W = \langle y_4\rangle$, $q(x_1) = x_2$, $q(x_2) = x_4 + y_4$, $q(x_4) = 0$, and $d(x_1) = y_4$. Then
$$\Gamma(q) = \Z/2[x_1,x_2, x_4, y_4]/(x_1^2 - x_2, x_2^2 - x_4 - y_4, x_4^2),$$ $$E_1(\Gamma(q)) = \Gamma(q_U) = \Z/2[x_1]/(x_1^8)\otimes \Z/2[y_4], \text{ and}$$ $$E^1E_1(\Gamma(q)) = \Gamma(0) = \Lambda(x_1,x_2,x_4) \otimes \Z/2[y_4].$$
The associated homology Hopf algebras are
$$H(\Gamma(q);d) = H(\Gamma(q_U);d) = \Z/2[x_2]/(x_2^4), \ H(\Gamma(0);d) = \Lambda(x_2, x_4).$$
\end{ex}

\begin{proof}[Proof of \propref{Hopf alg prop}]  We are given
$$ V_1 \xra{q} V_2 \xra{q} V_4 \xra{q} \dots,$$
and $V_{2^s} \xra{d} V_{2^t}$ such that $dq: V_{2^{s-1}} \ra V_{2^t}$ is zero. We also know that
\begin{equation*}\tag{$\diamondsuit_t$} q: V_{2^k} \ra V_{2^{k+1}} \text{ is monic for } k \geq t,
\end{equation*}

For $k>t$ choose a subspace $W_{2^k} \subset V_{2^k}$ such that $W_{2^k} \oplus q(V_{2^{k-1}}) = V_{2^k}$, and let $U = \bigoplus_{k<t} V_{2^k}$, $W = V_{2^t} \oplus \bigoplus_{k>t} W_{2^k}$.  Slightly abusing notation, $q$ and $d$ define maps $q: U \ra U \oplus W$ and $d: U \ra W$ satisfying $dq_U = 0$.  Condition ($\diamondsuit_t$) implies that
$(U_q(V),d) \simeq (\Gamma(q),d)$ as chain complexes.  \lemref{useful lemma} thus implies that
$$ H(U_q(V);d) \simeq \Gamma(q^{\prime}).$$
Now one observes that the righthand side rewrites as $U_q(\bar V)$, with $\bar V$ as in the proposition: the key point is that, since $W_{2^k} \oplus q(V_{2^{k-1}}) = V_{2^k}$, $W_{2^k} \oplus q(\bar V_{2^{k-1}}) = \bar V_{2^k}$.
\end{proof}

\subsection{The algebraic spectral sequence}

In this section we explain why, given $M \in \M$, there is a well defined spectral sequence of Hopf algebras as in \thmref{theorem 2}:
\begin{itemize}
\item $ E^{alg, 1}_{*,*}(M) = U_{\Qq}(\mathcal R_*(M))$.
\item Nonzero differentials are only the $d^{2^s}$, and, for $x \in M$ and $I$ of length $s$, $Q^Ix$ lives to $E^{alg, 2^s}_{*,*}(M)$, and
$d^{2^s}(Q^Ix) = \sum_{i \geq 0} Q^IQ^{i-1}(xSq^i)$.
\item For all $r$, $E^{alg,r}_{*,*}(M)$ is primitively generated with primitives concentrated in the $-2^s$ lines.  For all $r>2^s$, the module of primitives in $E^{alg,r}_{-2^s,2^s+ *}(M)$ is naturally isomorphic to $L_sM$.

\item  $ E^{alg, \infty}_{*,*}(M) = U_{\Qq}(L_*M)$.
\end{itemize}
(We recall that $L_sM = \Omega \Oinfty_s \Sigma^{1-s}M$.)

We do this by explicitly describing all the intermediate pages $E^{alg,2^s}_{*,*}(M)$, so that \propref{Hopf alg prop} applies.

In the rest of this subsection we use the following notation.

\begin{notation}  We let $H_sM$ be the homology at $\mathcal R_sM$ in the sequence
$ \mathcal R_{s-1}\Sigma^{-1} M \xra{d_{s-1}} \mathcal R_s M \xra{d_{s}} \mathcal R_{s+1} \Sigma M$.  We generically use `$q$' to denote maps induced by bottom Dyer--Lashof operations (a.k.a.\ squaring).  For $k\geq s$, we let
$\bar{\mathcal R}_{k,s}M = \mathcal R_kM/\im\{q^{k-s}d_{s-1}: \mathcal R_{s-1}\Sigma^{-1}M \ra \mathcal R_kM\}$.
\end{notation}

We note a few of the relationships between these:
\begin{itemize}
\item $0 \ra H_sM \ra \bar{\mathcal R}_{s,s}M \xra{d_s} \mathcal R_{s+1}\Sigma M$ is exact.
\item $L_sM = \im \{\Sigma H_s \Sigma^{-1}M \xra{\epsilon_*} H_sM\}$, by \corref{image-epsilon-loop}.
\item For $k\geq s$, the square
\begin{equation*}
\xymatrix{
\mathcal R_kM \ar[d] \ar[r]^-q & \mathcal R_{k+1}M \ar[d]  \\
\bar{\mathcal R}_{k,s}M \ar[r]^-q & \bar{\mathcal R}_{k+1,s}M }
\end{equation*}
is a pushout square with epic vertical map, thus defining the lower horizontal map.
\end{itemize}

\begin{lem}  $d_s: \mathcal R_s\Sigma^{-1}M \ra \mathcal R_{s+1}M$ induces a map
$d_s: \bar{\mathcal R}_{s,s}M \ra \bar{\mathcal R}_{s+1,s}M$.
\end{lem}
\begin{proof} The module $\bar{\mathcal R}_{s,s}M$ is the quotient of $\mathcal R_s\Sigma^{-1}M$ by the subspace generated by the images of the maps $d_{s-1}: \mathcal R_{s-1}\Sigma^{-2}M \ra \mathcal R_s\Sigma^{-1}M$ and $q: \mathcal R_{s-1}\Sigma^{-1}M \ra \mathcal R_s\Sigma^{-1}M$. Meanwhile $\bar{\mathcal R}_{s+1,s}M$ is the quotient of $\mathcal R_{s+1} M$ by the subspace given by the image of $d_{s}q: \mathcal R_{s-1}\Sigma^{-1}M \ra \mathcal R_{s+1}M$.  Clearly $d_s: \mathcal R_{s}\Sigma^{-1}M \ra \mathcal R_{s+1}M$ carries the former subspace to the latter.
\end{proof}

Armed with these various constructions, we can define our spectral sequence.

\begin{defn} Let $(E^{alg,2^s}_{*,*}(M),d^{2^s}) = (U_q(V(s)),d_s)$, where the pair $(V(s),d_s)$ is defined by
letting
\begin{equation*}
V(s)_{2^k} =
\begin{cases}
L_kM & \text{for } k<s \\  \bar{\mathcal R}_{k,s}M& \text{for } k \geq s,
\end{cases}
\end{equation*}
and by letting the nonzero component of $d_s$ be the map $d_s: \bar{\mathcal R}_{s,s}M \ra \bar{\mathcal R}_{s+1,s}M$ of the last lemma.
\end{defn}

\propref{Hopf alg prop} applies to prove $H(E^{alg,2^s}_{*,*}(M);d^{2^s}) = E^{alg,2^{s+1}}_{*,*}(M)$ once we check the next proposition.

\begin{prop}  There are natural isomorphisms
$$ \ker\{d_s: \bar{\mathcal R}_{s,s}M \ra \bar{\mathcal R}_{s+1,s}M\} \simeq L_sM$$
and
$$ \coker\{d_s: \bar{\mathcal R}_{s,s}M \ra \bar{\mathcal R}_{s+1,s}M\} \simeq \bar{\mathcal R}_{s+1,s+1}M.$$
\end{prop}

Proving this is slightly subtle.  We apply the next lemma, letting the diagram
\begin{equation*}
\xymatrix{
& \mathcal R_{s-1}\Sigma^{-1}M \ar[d]^{q_0} \ar[r]^-{d_{s-1}} & \mathcal R_{s}M \ar[d]^{q_0} \\
\mathcal R_{s-1}\Sigma^{-2}M \ar[d]^{\epsilon} \ar[r]^-{d_{s-1}} & \mathcal R_{s}\Sigma^{-1}M \ar[d]^{\epsilon} \ar[r]^-{d_s} & \mathcal R_{s+1}M \ar[d]^{\epsilon} \\
\mathcal R_{s-1}\Sigma^{-1}M \ar[d] \ar[r]^-{d_{s-1}} & \mathcal R_{s}M \ar[d] \ar[r]^-{d_s} & \mathcal R_{s+1}\Sigma M \ar[d] \\
0  & 0  & 0  }
\end{equation*}
play the role of ($\heartsuit$).

So suppose given a diagram of vector spaces
\begin{equation}\tag{$\heartsuit$}
\xymatrix{
& U \ar[d]^{q_U} \ar[r]^-{d_U} & V \ar[d]^{q_V} \\
U^{\prime} \ar[d]^{e_U} \ar[r]^-{d^{\prime}_U} & V^{\prime} \ar[d]^{e_V} \ar[r]^-{d^{\prime}_V} & W^{\prime} \ar[d]^{e_W} \\
U \ar[d] \ar[r]^-{d_U} & V \ar[d] \ar[r]^-{d_V} & W \ar[d] \\
0  & 0  & 0  }
\end{equation}
where the columns are exact, and the two bottom rows are chain complexes.

Let $\bar V = V/\im d_U$, $\bar W = W$, $\bar V^{\prime} = V^{\prime}/(\im d_U^{\prime} + \im q_U)$, and $\bar W^{\prime} = W^{\prime}/\im d_V^{\prime}\circ q_U$. The diagram ($\heartsuit$) maps in an evident way to the diagram
\begin{equation}\tag{$\spadesuit$}
\xymatrix{
& 0 \ar[d] \ar[r] & \bar V \ar[d]^{q} \\
0 \ar@{=}[d] \ar[r] & \bar V^{\prime} \ar[d]^{\bar e_{V}} \ar[r]^-{d^{\prime}} & \bar W^{\prime} \ar[d]^{\bar e_{W}} \\
0 \ar@{=}[d] \ar[r] & \bar V \ar[d] \ar[r]^-{d} & \bar W \ar[d] \\
0  & 0  & 0.  }
\end{equation}

\begin{lem} \label{sneaky lemma} The following properties hold.\\

\noindent{\bf (a)} $\bar e_V: \bar V^{\prime} \ra \bar V$ is an isomorphism. \\

\noindent{\bf (b)} $\ker (d^{\prime}) = \im\{e_{V_*}:H(V^{\prime}) \ra H(V)\}$. \\

\noindent{\bf (c)} $\coker(d^{\prime}) \simeq \coker(d_V^{\prime})$. \\

\noindent{\bf (d)} The third column of $(\spadesuit)$ is exact. \\
\end{lem}

\begin{proof}  Diagram chasing with the left two columns of ($\heartsuit$) shows that there is an exact sequence
$$ U \xra{q_U} V^{\prime}/\im d_U^{\prime} \xra{e_V} V/\im d_U \ra 0,$$
and statement (a) follows.

It is standard that given maps $A \xra{f} B \xra{g} C$ in an abelian category, there is an exact sequence $\coker f \ra \coker gf \ra \coker g \ra 0$.  Apply this to $U \xra{d_U} V \xra{q_V} W^{\prime}$ to deduce statement (d).  Apply this to $V^{\prime} \xra{d_V^{\prime}} W^{\prime} \twoheadrightarrow \bar W^{\prime}$ to deduce statement (c), noting that
$$\coker\{V^{\prime} \ra \bar W^{\prime}\} = \coker\{\bar V^{\prime} \xra{d^{\prime}} \bar W^{\prime}\}.$$

To deduce (b), let $\tilde V^{\prime} = V^{\prime}/\im d_U^{\prime}$. One has a commutative diagram
\begin{equation*}
\xymatrix{
\tilde V^{\prime} \ar[d]^{d_V^{\prime}} \ar[r] & \bar V^{\prime} \ar[d]^{d^{\prime}} \ar[r]^{\sim} & \bar V \ar[d]^d \\
W^{\prime}  \ar[r] & \bar W^{\prime}  \ar[r] & \bar W,  }
\end{equation*}
where the indicated isomorphism is the isomorphism of (a).  Taking kernels, one gets
$$ H(V^{\prime}) \ra \ker d^{\prime} \hra H(V),$$
and we need to check that the first map here is onto.  But this follows because the left square fits into a commutative diagram
\begin{equation*}
\xymatrix{
U \ar[r] \ar@{=}[d] &\tilde V^{\prime} \ar[d]^{d_V^{\prime}} \ar[r] & \bar V^{\prime}  \ar[d]^{d^{\prime}} \ar[r] & 0 \\
U \ar[r] & W^{\prime}  \ar[r] & \bar W^{\prime} \ar[r] & 0   }
\end{equation*}
with exact rows.
\end{proof}

\section{Examples} \label{example section}

\subsection{Eilenberg--MacLane spectra}
We begin our discussion of how the spectral sequence behaves when $X = \Sigma^n HA$ by noting that all of our constructions behave well with respect to filtered colimits, localization at 2,  and direct sums in the variable $A$.  It follows that the key cases to understand are when $A = \Z$, $\Z/2$, and $\Z/2^r$ with $r\geq 2$.

Our first task is compute $L_*(H_*(\Sigma^n HA))$ in these cases.

Recall that $H^*(H\Z/2) = \A$, $H^*(H\Z) = \A/\A Sq^1$, and $H^*(H\Z/2^r) = H^*(H\Z) \oplus \Sigma H^*(H\Z)$ for $r \ge 2$.  For convenience, let $\bar \A_* = H_*(H\Z)$.

\begin{lem} For all $s>0$,
$\Oinfty_s \Sigma^{1-s+n} \bar \A_* =
\begin{cases}
\Sigma \Z/2 & \text{if } n = 0 \\
 \Z/2 & \text{if } n = -1 \\
0 & \text{otherwise}.
\end{cases}
$
\end{lem}
\begin{proof}  We work with the equivalent dual left $\A$--module situation.  Let $F(n) = \Oinfty \Sigma^n \A$, the free unstable $\A$--module on an $n$--dimensional class.  (This is 0, if $n<0$.)  The module $\A/\A Sq^1$ has a projective resolution
$$ \cdots \ra \Sigma^2 \A \xra{\cdot Sq^1} \Sigma \A \xra{\cdot Sq^1}  \A \ra \A/\A Sq^1 \ra 0.$$
Applying $\Oinfty \Sigma^{1-s+n}$ yields the complex
$$ \cdots \ra F(1-s+n+2) \xra{\cdot Sq^1} F(1-s+n+1) \xra{\cdot Sq^1}  F(1-s+n). $$
The module $\Oinfty_s \Sigma^{1-s+n} \bar \A_*$ is thus dual to the homology of
$$ F(n+2) \xra{\cdot Sq^1} F(n+1) \xra{\cdot Sq^1} F(n).$$
By inspection, one sees that this is exact except when $n = -1$ or 0.
\end{proof}

\begin{cor} {\bf (a)} $L_sH_*(\Sigma^n H\Z/2) = 0$ for all $s>0$ and all $n$. \\

\noindent{\bf (b)} For all $s>0$, $L_sH_*(\Sigma^n H\Z) = \begin{cases}
\Z/2 & \text{if } n = 0 \\
0 & \text{otherwise}.
\end{cases}$ \\

\noindent{\bf (c)} For all $s>0$ and $r\geq 2$,
$$L_sH_*(\Sigma^n H\Z/2^r) = L_sH_*(\Sigma^n H\Z) \oplus L_sH_*(\Sigma^{n+1} H\Z) = \begin{cases}
\Z/2 & \text{if } n = -1, 0 \\
0 & \text{otherwise}.
\end{cases}$$
\end{cor}

Now we need to know how the Dyer--Lashof operation $Q^0$ acts.

\begin{lem} $Q^0: L_s H_*(H\Z) \ra L_{s+1} H_*(H\Z)$ is an isomorphism for all $s \geq 0$.
\end{lem}
\begin{proof} The key point is that the exact sequence
$$ \Oinfty_s \Sigma^{1-s}\bar \A_* \xra{sq_0} \Phi(\Oinfty_s \Sigma^{1-s}\bar \A_*) \xra{q_0} \Sigma \Oinfty_{s+1} \Sigma^{-s}\bar \A_*$$
identifies with the exact sequence
$$ \Sigma \Z/2 \xra{sq_0} \Sigma^2 \Z/2 \xra{q_0} \Sigma^2 \Z/2.$$
As the first map here is clearly zero, the second is an isomorphism.
\end{proof}

This lemma and the previous corollary combine to give us the next calculations.

\begin{prop} \label{UL prop} {\bf (a)} $U_{\Qq}(L_*H_*(H\Z)) = \Z/2[x]$ where $x$ is the nonzero 0 dimensional class in  $\Oinfty H_*(H\Z)$. \\

\noindent {\bf (b)} For $r\geq 2$, $U_{\Qq}(L_*H_*(H\Z/2^r)) = \Z/2[x] \otimes \Lambda^*(y)$ where $x$ and $y$ are the nonzero 0 and 1 dimensional classes in $\Oinfty H_*(H\Z/2^r)$. \\

\noindent {\bf (c)} For $r\geq 2$, $U_{\Qq}(L_*H_*(\Sigma^{-1}H\Z/2^r)) = \Z/2[y]$ where $y$  is the nonzero 0 dimensional class in $\Oinfty \Sigma^{-1}H_*(H\Z/2^r)$. \\

\noindent{\bf (d)} If $A$ is either $\Z$ or $\Z/2^r$, then $U_{\Qq}(L_*H_*(\Sigma^{n}HA)) = \Lambda^*(\Oinfty H_*(\Sigma^nHA))$, except in the cases covered by (a), (b), and (c).
\end{prop}

By inspection we see that the algebraic condition of \corref{alg = top cor 2} holds for all Eilenberg--MacLane spectra.

\begin{cor}  For all abelian groups $A$ and all $n \in \Z$, $L_*H_*(\Sigma^{n}HA)$ is generated as a module over the Dyer--Lashof algebra by $L_0H_*(\Sigma^{n}HA) = \Oinfty H_*(\Sigma^{n}HA)$.
\end{cor}

The classic calculations of $H^*(K(A,n))$ when $A$ is either $\Z$ or $\Z/2^r$ allow us to determine when the geometric condition of \corref{alg = top cor 2} holds.

\begin{lem} If $A$ is either $\Z$ or $\Z/2^r$, the evaluation map
$$  H_*(K(A,n)) \ra \Oinfty H_*(\Sigma^{n}HA)$$
is onto except when $A = \Z/2^r$ with $r\geq 2$ and $n = 0$ or $-1$.
\end{lem}

We conclude that the algebraic spectral sequence equals the topological spectral sequence for all Eilenberg--MacLane spectra $\Sigma^n HA$, unless $A$ has 2--torsion of order at least 4, and $n = 0$ or $-1$.

As a check of our work, note that our spectral sequence correctly tells us that, up to filtration, $H_*(K(\Z/2,n))$ equals $\Lambda^*(F(n)_*)$ with no nonzero Dyer--Lashof operations.

Now we use our calculations to say more about how the topological spectral sequence behaves in the cases covered by  \propref{UL prop}(a), (b), and (c).

\subsection{Convergence of the spectral sequence for $H\Z$}  We have shown that $E^{\infty}_{*,*}(H\Z) = \Z/2[x]$.  The spectral sequence converges to the correct answer as well as possible: $H_*(\Oinfty H\Z) = H_*(\Z) = \Z/2[t,t^{-1}]$ while $\displaystyle \lim_d H_*(P_d(H\Z)) = \Z/2[[x]]$, and the former embeds densely in the latter via the homomorphism sending $t$ to $x+1$.

\subsection{The spectral sequence for $H\Z/2^r$ with $r\geq 2$} When $r\geq 2$, we have shown that $E^{alg,\infty}_{*,*}(H_*(H\Z/2^r)) = \Z/2[x] \otimes \Lambda^*(y)$.  This time only $x$ is in the image of the evaluation, so there should be a rogue differential off of $y$.  The elements $x^{2^s}$ are the only nonzero 0 dimensional primitive classes in $E^1$, so the first rogue differential must hit one of these.

We claim that $d^{2^r-1}(y) = x^{2^r}$, this is the only rogue differential, and $E^{\infty}_{*,*}(H\Z/2^r) = \Z/2[x]/(x^{2^r})$.  Furthermore, the spectral sequence converges to the correct answer: $H_*(\Oinfty H\Z/2^r) = H_*(\Z/2^r) = \Z/2[t]/(t^{2^r}-1) = Z/2[x]/(x^{2^r})$, when $t=x+1$.

To prove the claim, we first make some observations about the beginning of the spectral sequence in low degrees.  In total degree 0, $E^1$ is spanned by the classes $x^n$, and in total degree 1, $E^{1}$ is spanned by the classes $x^ny$, and $x^nQ^1x$.  If $z \in H_*(H\Z/2^r)$ is the two dimensional class with $zSq^2 =x$, then $d^1(x^nz) = x^nQ^1x$.  It follows that the only classes in $E^2$ in degrees 0 and 1 will be $x^n$ and $x^ny$, none of which can possibly be in the image of an algebraic differential.

We now show that $x^{2^r} = 0$ in $E^{\infty}_{*,*}(H\Z/2^r)$.  To see this, we consider the diagram
\begin{equation*}
\xymatrix{
Z/2[t,t^{-1}] \ar[d] \ar[r] & \Z/2[[x]] \ar[d]  \\
Z/2[t]/(t^{2^r}-1) \ar[r] & \lim_d H_*(P_d(H\Z/2^r)) }
\end{equation*}
in which both horizontal maps send $t-1$ to $x$.  As $(t-1)^{2^r} = t^{2^r}-1 = 0$ in $Z/2[t]/(t^{2^r}-1)$, we see that $x^{2^r} = 0$ in $\displaystyle \lim_d H_*(P_d(H\Z/2^r))$, and thus in $E^{\infty}_{*,*}(H\Z/2^r)$.

Finally we show that  $x^{2^s} \neq 0$ for all $s<r$, or equivalently, that $y$ lives to $E^{2^{r}-1}$.  This we show by induction on $r$.  The $r=2$ case is true because $d^1(y)=0$. For the inductive step, let $E^{alg,\infty}_{*,*}(\Z/2^{r-1}) = \Z/2[x^{\prime}] \otimes \Lambda^*(y^{\prime})$.  The inclusion $\Z/2^{r-1} \ra \Z/2^r$ induces a map of both the topological and algebraic spectral sequences sending $x^{\prime}$ to 0, and $y^{\prime}$ to $y$.  Then the inductive hypothesis --- that $y^{\prime}$ lives to $E^{2^{r-1}-1}_{*,*}(H\Z/2^{r-1})$ and $d^{2^{r-1}-1}(y^{\prime}) = (x^{\prime})^{2^{r-1}}$ --- implies that $y$ lives to $E^{2^{r-1}-1}_{*,*}(H\Z/2^{r})$ and $d^{2^{r-1}-1}(y) = 0$, i.e.\ $y$ lives to $E^{2^{r-1}}_{*,*}(H\Z/2^{r})$, and thus to $E^{2^{r}-1}_{*,*}(H\Z/2^{r})$.

\subsection{The spectral sequence for $\Sigma^{-1}H\Z/2^r$ with $r\geq 2$} \label{crazy ex subsection} Our most complicated example is the spectral sequence for $\Sigma^{-1}H\Z/2^r$, with $r\geq 2$.

Let $x$ and $y$ be the nonzero classes in $H_*(\Sigma^{-1}H\Z/2^r)$ of dimensions $-1$ and 0.
$E^{alg,\infty}_{*,*}(H_*(\Sigma^{-1} H\Z/2^r)) = \Z/2[y]$, and obviously $y$ is not in the image of the evaluation.
The only primitive elements in $E^1$ of total degree $-1$ are the elements $(Q^0)^sx \in E^1_{-2^s,2^s-1}$, so a first rogue differential must hit one of these.

We claim that $y$ lives to $E^{2^r}$, and $d^{2^r-1}(y) = (Q^0)^rx$.  To see this, we compare this example to our previous one, using the map of spectral sequences induced by
$$ \Sigma P(\Sigma^{-1}H\Z/2^r) \ra P(H\Z/2^r).$$
This sends the elements $x$ and $y$ to the elements with the same name in the last example.  It also induces an isomorphism from the primitives of total degree $-1$ in $E^1(\Sigma^{-1}H\Z/2^r)$ to the primitives of total degree 0 in $E^1(H\Z/2^r)$.  The calculation that $d^{2^r-1}(y) = x^{2^r} = (Q^0)^rx$ in the spectral sequence for $H\Z/2^r$ then implies that $d^{2^r-1}(y) = (Q^0)^rx$ in the spectral sequence for $\Sigma^{-1}H\Z/2^r$.

The formula $d^{2^r-1}(y) = (Q^0)^rx$ then implies that, for any $s \geq 0$,
$$d^{2^s(2^r-1)}(y^{2^s}) = d^{2^s(2^r-1)}((Q^0)^sy) = (Q^0)^sd^{2^r-1}(y) = (Q^0)^{s+r}x.$$

We also note that $d^1(x) = Q^{-1}x = x^2$, and it follows that, for any $s\geq 0$,
$$ d^{2^s}((Q^0)^sx) = (Q^0)^sQ^{-1}x = Q^{-1}(Q^0)^sx = ((Q^0)^sx)^2.$$

We now explain how these calculations completely determine how the algebraic and topological spectral sequences differ.  Let $x_s = (Q^0)^sx$.  Using the standard primitive generators, the $E^1$ term of both spectral sequences decomposes:
$$ E^1 = \Z/2[y,x_0,x_1,x_2, \dots] \otimes E^{\bot, 1}.$$
This, in fact, represents a decomposition of both spectral sequences, where the algebraic and topological spectral sequences agree on $E^{\bot,*}$, and the differentials on $\Z/2[y,x_0,x_1,x_2, \dots]$ go as follows:
\begin{itemize}
\item The algebraic spectral sequence has $d^{2^s}(x_s) = x_s^2$.
\item The topological spectral sequence also has $d^{2^s(2^r-1)}(y^{2^s}) = x_{s+r}$.
\end{itemize}
It is then easy to compute that, for all $s\geq 0$,
$$ E^{alg, 2^s} = \Z/2[y,x_s,x_{s+1}, x_{s+2}, \dots] \otimes E^{\bot,2^s},$$
while, for all $s\geq r$,
$$ E^{top, 2^s} = \Z/2[y^{2^{s+1-r}},x_{s+1}, x_{s+2}, \dots] \otimes E^{\bot,2^s}.$$

\subsection{The spectral sequence for suspension spectra}  Let $X = \Sinfty Z$, a suspension spectrum, so that $H_*(X)$ is unstable. The tower is known to split; for example, when $Z$ is connected,
$$ \Sinfty \Oinfty \Sinfty Z \simeq \bigvee_d \Sinfty D_d Z.$$ Thus the spectral sequence collapses at $E^1$, and so
$$ E^{\infty}_{*,*}(X) \simeq U_{\Qq}(\mathcal R_*H_*(X)).$$

As we clearly have no rogue differentials, our works says that
$$E^{\infty}_{*,*}(X) \simeq U_{\Qq}(L_*H_*(X)).$$

This is in agreement with \thmref{unstable thm}(a), which says that, since $H_*(X)$ is unstable,  $L_*H_*(X) = \mathcal R_*H_*(X)$.

\subsection{The spectral sequence for $S^1\langle 1 \rangle$}  The partially published work of Lannes and Zarati \cite{lz2} suggests that one might be able to `mix and match' the suspension spectra and Eilenberg--MacLane spectra examples.  Here is the simplest such example.

Let $S^1\langle 1\rangle$ be the cofiber of $S \ra H\Z$.

By dimension shifting, one can easily compute that, for all $s\geq 0$,
$$L_sH_*(S^1\langle 1\rangle) \simeq L_{s+1}H_*(S^1) = \mathcal R_{s+1}H_*(S^1),$$
and this is compatible with Dyer--Lashof operations.

Meanwhile, there is a map $\Sinfty D_2S^1 \xra{t} S^1\langle 1\rangle$ such that $t_*$ realizes the isomorphism $H_*(D_2S^1) = \mathcal R_1H_*(S^1) \simeq \Oinfty H_*(S^1\langle 1\rangle)$: $t$ can be taken to be the bottom horizontal map in the commutative square of symmetric products
\begin{equation*}
\xymatrix{
SP^2(S) \ar[d] \ar[r] & SP^{\infty}(S) \ar[d]  \\
SP^2(S)/S \ar[r] & SP^{\infty}(S)/S.}
\end{equation*}
One formally concludes that $H_*(\Oinfty S^1\langle 1\rangle) \ra \Oinfty H_*(S^1\langle 1\rangle)$ is onto.

\corref{alg = top cor 2} now applies to say that
$E^{\infty}_{*,*}(S^1\langle 1\rangle) = \U_{\Qq}(\mathcal R_{*\geq 1}H_*(S^1))$.

This is in agreement with known calculation: $\Oinfty S^1\langle 1\rangle$ is the fiber of the split fibration $\Oinfty \Sinfty S^1 \ra S^1$, and it is not hard to  see that, localized at 2,
$$ \Sinfty \Oinfty S^1\langle 1\rangle \simeq \bigvee_{d} D_{2d}S^1.$$

\subsection{Adams resolutions of suspension spectra}  Here is a much more sophisticated version of the last example.  Let $Z$ be a connected space, and, for $s \geq 0$, recursively define spectra $Z(s)$ and $K(s)$ by letting $Z(0) = \Sinfty Z$, $K(s) = Z(s) \sm H\Z/2$, and $Z(s) = \hofib\{Z(s-1) \xra{i} K(s-1)\}$.

In \cite{lz2}, an only partially finished manuscript from the 1980's but supported by \cite{lz}, Lannes and Zarati study $H^*(\Oinfty Z(s))$.  Enroute, they show (in dual form) \cite[Prop.2.5.2(iv)]{lz2} that $H_*(\Oinfty Z(s)) \ra \Oinfty H_*(Z(s))$ is onto. (This is proved by a rather elaborate induction on $s$ using the Eilenberg--Moore spectral sequence.)

Meanwhile, dimension shifting immediately shows that, for $t>0$, there are isomorphisms
$$ L_tH_*(Z(s)) \simeq \mathcal R_{s+t}H_*(Z).$$

We claim that there is also an epimorphism $L_0H_*(Z(s)) \ra \mathcal R_{s}H_*(Z)$, so that $L_0H_*(Z(s))$ generates $L_*H_*(Z(s))$, and once again \corref{alg = top cor 2} applies.

To see this, we begin with the exact sequence
\begin{multline*}
0 \ra \Oinfty H_*(Z(s-1)) \xra{i_*} \Oinfty H_*(K(s-1)) \\
\ra \Oinfty \Sigma H_*(Z(s)) \ra  \Oinfty_1 H_*(Z(s-1)) \ra 0.
\end{multline*}

Let $M$ denote the cokernel of $i_*$, so there is a short exact sequence
$$ 0 \ra M \ra  \Oinfty \Sigma H_*(Z(s)) \ra  \Oinfty_1 H_*(Z(s-1)) \ra 0.$$

Applying $\Omega$ to this yields an exact sequence
$$ 0 \ra \Omega M \ra L_0H_*(Z(s)) \ra R_sH_*(Z) \ra \Omega_1 M \ra \dots.$$
Now we note that $\Omega_1M=0$.  \lemref{omega 1 lemma} tells us that $\Sigma \Omega_1 M$ is the cokernel of $sq_0: M \ra \Phi(M)$.  But this map is onto, as it fits into a square
\begin{equation*}
\xymatrix{
\Oinfty H_*(K(s-1)) \ar[d]^{sq_0} \ar[r] & M \ar[d]^{sq^0}  \\
\Phi(\Oinfty H_*(K(s-1))) \ar[r] & \Phi(M) }
\end{equation*}
in which the horizontal maps are onto by construction, and the left vertical arrow is onto by \lemref{reduced lemma}.

To conclude,  $E^{\infty}_{*,*}(Z(s)) = \U_{\Qq}(L_*H_*(Z(s)))$, and this fits into a short exact sequence in $\HQU$:
$$\Z/2  \ra \Lambda^*(\Omega M) \ra \U_{\Qq}(L_*H_*(Z(s))) \ra \U_{\Qq}(\mathcal R_{*\geq s} H_*(Z)) \ra \Z/2.$$

\subsection{A rogue differential for a 0--connected finite complex} \label{BG example subsection}

Here is perhaps the simplest example of a rogue differential occurring in the spectral sequence of a $0$--connected spectrum $X$.

Let the spectrum $X$ be the cofiber of $4: \R P^4 \ra \R P^4$, so that $X$ fits into a cofibration sequence
$$ \R P^4 \ra X \ra \Sigma \R P^4.$$
As $4$ has Adams filtration 2, we are guaranteed that
$$H_*(X) \simeq H_*(\R P^4 \vee \Sigma \R P^4) \simeq H_*(\R P^4) \oplus \Sigma H_*(\R P^4),$$ as right $\A$--modules.  For $i = 1,2,3,4$, let $a_i \in H_i(X)$ be the image of the nonzero element under the inclusion $\R P^4 \hra X$, and let $b_i \in H_{i+1}(X)$ project to a nonzero element under the projection $X \ra \Sigma \R P^4$.

As $H_*(X) \in \U$, if there were no rogue differentials, then $E^{\infty}_{*,*}(X) = E^{1}_{*,*}(X)$.  We show this is impossible.

\begin{prop} \label{example prop}In the spectral sequence, $d^3(b_4) = a_1^4$.
\end{prop}

Before proving this, we note some properties that $X$ must (not) have.

\begin{lem} $X$ is not homotopy equivalent to $\R P^4 \vee \Sigma \R P^4$.
\end{lem}
\begin{proof} This follows easily from the fact that the identity on $\R P^4$ has stable order 8, not 4 \cite{toda}.
\end{proof}

\begin{cor} The evaluation $H_*(\Oinfty X) \ra H_*(X)$ is not onto.
\end{cor}
\begin{proof}  $\R P^4 \vee \Sigma \R P^4$ is the wedge of two (dual) Brown--Gitler spectra, and thus is homotopy equivalent to any other 2--complete connective spectrum $Y$ with isomorphic mod 2 homology such that $H_*(\Oinfty Y) \ra H_*(Y)$ is onto \cite{hunterkuhn}.
\end{proof}

\begin{proof}[Proof of \propref{example prop}] Figure 1 shows the $-1$ line, and the bottom nonzero elements in the next few lines, of $E^1_{*,*}(X)$ for the spectral sequence converging to $H_*(\Oinfty X)$.

Recalling that $d^1 \equiv 0$, and that differentials take primitives to primitives, the only possible nonzero differential off of the $-1$ line would be $d^3(b_4) = a_1^4$.
Thus if  $d^3(b_4) = a_1^4$ did {\em not} hold, then we could conclude that $E^{\infty}_{-1,*}(X) = E^{1}_{-1,*}(X)$, so that $H_*(\Oinfty X) \ra H_*(X)$ would be onto, contradicting the corollary.

\begin{figure}
\begin{equation*}
\begin{array}{ccccc|c}
a_1^4& &&&& 8\\
& &&&& 7\\
 &a_1^3 &&\text{\hspace{.2in}}b_4&& 6\\
&   &&a_4,b_3&& 5\\
 &   &a_1^2&a_3,b_2&& 4\\
 &    &&a_2,b_1&& 3\\
&     &&a_1 \text{\hspace{.17in}}&& 2\\
& &&&& 1\\
& &&&  1   & 0  \\ \hline
-4& -3&-2&-1& 0&s\backslash t \\
\end{array}
\end{equation*}
\caption{$E^1_{s,t}(X)$} \label{figure 1}
\end{figure}

\end{proof}

\subsection{The spectral sequence for $S^1 \cup_{\eta} D^3$}  \label{CP2 example} Let $X = S^1 \cup_{\eta} D^3 = \Sigma^{-1} \Sinfty \mathbb C P^2$, so the spectral sequence is converging to $H_*(\Omega \Oinfty \Sinfty \mathbb C P^2)$.  Then
$H_*(X) = \langle x, y \rangle$ with $|x|=1$, $|y|=3$, and $ySq^2 = x$. Thus $d^1(y) = Q^1x$, and so $d^2(Q^3y) = Q^3Q^1x$, which is zero by the Dyer--Lashof Adem relations.  It follows that $Q^3y$ is an element in $L_1H_*(X)$ that is {\em not} in the Dyer--Lashof algebra module generated by $L_0H_*(X) = \langle x \rangle$.  Thus this is an example where the algebraic condition of \corref{alg = top cor 2} fails to hold, even while the geometric condition clearly does.

Another aspect of our theory easily seen here is that, though $E^1$ and $E^{\infty}$ support Dyer--Lashof operations, the in between pages needn't.  For example, $Q^1x \in B^2$, but $Q^4Q^1x = Q^3Q^2x \not \in B^2$.
\appendix
\section{Proof of \propref{S-mod prop}} \label{spectra appendix}

We need to explain the last property of $S$--modules listed in \propref{S-mod prop}.  This said that, given an $S$--module $X$, there is a weak natural equivalence $$\displaystyle \hocolim_n \Sigma^{-n}\Sinfty X_n \ra X.$$

We thank Mike Mandell for helping us be accurate in the following discussion.

Let $\Sinfty_n: \T \ra \text{Spectra}$ be left adjoint to $X \rightsquigarrow X_n$.  Recall that an $S$--module is a special sort of $\sL$--module.  The functor sending a spectrum $X$ to the $S$--module $S\sm_{\sL}\bL X$ is left adjoint to the functor sending an $S$--module $X$ to $F_{\sL}(S,X)$, just regarded as a spectrum (and not as an $\sL$--module).

There is a weak equivalence of $S$--modules
$$ S\sm_{\sL}\bL(\Sinfty_nX_n) \ra \Sigma^{-n} \Sinfty X_n $$
given as the adjoint to the composite of maps of spectra
$$ \Sinfty_nX_n \ra \Sigma^{-n} \Sinfty X_n \ra \Sigma^{-n} F_{\sL}(S,\Sinfty X_n) =  F_{\sL}(S,\Sigma^{-n}\Sinfty X_n).$$

There is a map of $S$--modules
$$ S\sm_{\sL}\bL(\Sinfty_nX_n) \ra X$$
given as the adjoint to the composite of maps of spectra
$$ \Sinfty_nX_n \ra \Sinfty_nF_{\sL}(S,X)_n \ra  F_{\sL}(S,X).$$

The desired weak natural equivalence is obtained by taking the hocolimit over $n$ of the zig-zag
$$ \Sigma^{-n} \Sinfty X_n \xla{\sim} S\sm_{\sL}\bL(\Sinfty_nX_n) \ra X$$

We note that the $n=0$ case of the zig-zag here has the form
$$ \Sinfty \Oinfty X \xla{\sim} S\sm_{\sL}\bL(\Sinfty \Oinfty X) \ra X,$$
which induces the evaluation (counit) map in the homotopy category.

\section{The tower $P(X)$ with its operad action} \label{tower appendix}

We explain how the results of \cite{ak} show that the operad $\C_{\infty}$ acts suitably on the tower $P(X)$ as described in \thmref{operad action on tower thm}.

The paper \cite{ak} explored the explicit model from \cite{arone} for the tower associated to the functor sending a space $Z$ to the spectrum $\Sinfty_+ \Map_\T(K,Z)$, where $K$ is a fixed CW complex.  Call this tower $P(K,Z)$, indicating its functoriality in both variables. (The more awkward notation $P^K(X)$ was used in \cite{ak}.)  It comes with a natural transformation $e: \Sinfty_+ \Map_\T(K,Z) \ra P(K,Z)$ which is an equivalence if the dimension of $K$ is less than the connectivity of $Z$.

We note that the properties of our category of spectra needed to form our constructions correspond to the first five properties of $\Sp$ listed in \propref{S-mod prop}.

The product theorem, \cite[Thm.1.4]{ak}, says that there is a weak natural equivalence of towers
$$ P(K \vee L,Z) \xra{\sim} P(K) \sm P(L).$$
This generalizes to more than two factors in a straightforward way.  In particular, if $\bigvee_d K$ denotes the wedge of $d$ copies of $K$, there is a $\Sigma_d$--equivariant map of towers of spectra
$$ P(\bigvee_d K, X) \ra P(K)^{\sm d}$$
which is a nonequivariant equivalence.

Specialized to $K=S^n$, one gets a tower $P(S^n,Z)$ approximating $\Sinfty_+ \Omega^n Z$ with $d$th fiber naturally weakly equivalent to $\C_n(d)_+ \sm_{\Sigma_d}(\Sigma^{-n}Z)^{\sm d}$, as expected.  Here $\C_n$ is the little $n$--cubes operad.

The naturality and continuity of the $P(K,Z)$ construction in the variable $K$ make it quite easy to define maps of towers
$$ \Theta(d): \C_n(d)_+ \sm_{\Sigma_d} P(\bigvee_d S^n,Z) \ra P(S^n,Z)$$
compatible with the usual $\C_n$ operad action  on $\Omega^n Z$ \cite{may}.  In particular, from \cite[Thm.1.10]{ak}, we learn that the square in the diagram
\begin{equation*}
\xymatrix{
& \Sinfty_+ \C_n(d) \times_{\Sigma_d} (\Omega^n Z)^d \ar[d]^{\C_n(d)_+ \sm_{\Sigma_d} e^{\sm d}} \ar[r]^-{\Theta(d)} & \Sinfty_+ \Omega^n Z \ar[d]^e  \\
\C_n(d)_+ \sm_{\Sigma_d} P(S^n,Z)^{\sm d} & \C_n(d)_+ \sm_{\Sigma_d} P(\bigvee_d S^n,Z)\ar[r]^-{\Theta(d)} \ar[l]_-{\sim}& P(S^n,Z) }
\end{equation*}
commutes.  Furthermore, the map on fibers induced by the map of towers corresponds to the maps induced by the operad structure in the expected way.

Given a spectrum $X$, our tower is then defined to be
$$ P(X) = \hocolim_n P(S^n, X_n),$$
where the homotopy colimit is over natural transformations
$$ P(S^n,X_n) \xra{\sm} P(S^{n+1},\Sigma X_n) \ra P(S^{n+1}, X_{n+1}).$$
Here the first map is the smashing map from \cite[Thm.1.1]{ak}.

The $d$th fiber of the tower $P(X)$ then naturally identifies with
$$\hocolim_n  \C_n(d)_+ \sm_{\Sigma_d}(\Sigma^{-n}\Sinfty X_n)^{\sm d} \simeq \C_{\infty}(d)_+ \sm_{\Sigma_d}X^{\sm d} = D_dX.$$

Finally the weak natural transformation $e: \Sinfty_+ \Oinfty X \ra P(X)$ is defined as the composite
$$ \Sinfty_+ \Oinfty X \xla{\sim} \hocolim_n \Sinfty_+ \Omega^n X_n \xra{\hocolim_n e} \hocolim_n P(S^n, X_n),$$
and the diagram of \thmref{operad action on tower thm} is obtained by taking the hocolimit over $n$ of diagrams as above (with $d$ specialized to 2).

%
%
%
%

\end{document}